\documentclass[11pt]{article}

\usepackage[utf8]{inputenc}
\usepackage[hidelinks]{hyperref}
\usepackage{booktabs}
\usepackage{amssymb}
\usepackage{amsmath}
\usepackage{rotating}
\usepackage{xcolor,colortbl}
\usepackage{amsthm}
\usepackage{authblk}
\usepackage{subcaption}
\usepackage{graphicx}
\usepackage{cleveref}
\usepackage{tikz}
\usepackage{xspace}
\usepackage{fullpage}
\usepackage{orcidlink}

\usepackage{natbib}
\usepackage{tabularx}
\usepackage{mathabx}
\usepackage{stmaryrd}
\usepackage{stackengine}

\bibpunct[, ]{(}{)}{,}{a}{}{,}%
\providecommand{\keywords}[1]
{
	\small	
	\textbf{\textit{Keywords---}} #1
}

\allowdisplaybreaks

\newtheorem{theorem}{Theorem}

\newtheorem{corollary}[theorem]{Corollary}

\newcommand{\assumptionautorefname}{Assumption}
\crefname{assumption}{\assumptionautorefname}{\assumptionautorefname{}s}

\newcommand{\ProblemName}{Rolling Stock Scheduling Problem\xspace}
\newcommand{\ProblemAcronym}{RSSP}
\newcommand{\CompositionModel}{Composition model\xspace}
\newcommand{\HypergraphModel}{Hypergraph model\xspace}

\newcommand{\TripSet}{T}

\newcommand{\TypeSet}{R}

\newcommand{\N}{\mathbb{N}}
\newcommand{\R}{\mathbb{R}}

\newcommand{\Sum}{\sum\limits}

\newcommand{\PIP}{P_{\text{\normalfont IP}}}
\newcommand{\PLP}{P_{\text{\normalfont LP}}}
\newcommand{\valIP}{\nu_{\,\text{\normalfont IP}}}
\newcommand{\valLP}{\nu_{\,\text{\normalfont LP}}}
\newcommand{\valMP}{\nu_{\,\text{\normalfont MP}}}

\newcommand{\olA}{{\overline{A}}}

\newcommand{\sign}{\mathop\text{\rm sign}}
\newcommand{\tick}{\mathop\text{\rm tick}}
\newcommand{\tock}{\mathop\text{\rm tock}}

\newcommand{\IPC}{{\rm IP_C}}
\newcommand{\IPM}{{\rm IP_M}}

\newcommand{\HD}{\text{\normalfont HD}\xspace}
\newcommand{\HA}{\text{\normalfont HA}\xspace}
\newcommand{\HolA}{{\text{\normalfont H}\overline{\text{\normalfont A}}\xspace}}
\newcommand{\hD}{\text{\normalfont hD}\xspace}
\newcommand{\hA}{\text{\normalfont hA}\xspace}
\newcommand{\holA}{{\text{\normalfont h}\overline{\text{\normalfont A}}\xspace}}
\newcommand{\C}{\text{\normalfont C}\xspace} 
\newcommand{\LP}{\text{\normalfont LP}\xspace}

\newcommand{\IP}{\text{\normalfont IP}\xspace}
\newcommand{\MP}{\text{\normalfont MP}\xspace}
\def\permille{\ensuremath{{}^\text{o}\mkern-5mu/\mkern-3mu_\text{oo}}}

\usepackage{pgfplotstable}
\pgfplotsset{compat=1.15}

\title{A Comparison of Models for Rolling Stock Scheduling}
\author[a]{Boris Grimm,\orcidlink{0009-0005-8080-9663}\,}
\author[b]{Rowan Hoogervorst\,\orcidlink{0000-0003-0358-9503}\,\thanks{A majority of the work was done while the author was working at Erasmus University Rotterdam, the Netherlands}}
\author[a]{Ralf Bornd\"{o}rfer\,\orcidlink{0000-0001-7223-9174}\,}
\affil[a]{Zuse Institute Berlin (ZIB), 14195 Berlin, Germany}
\affil[b]{Department of Technology, Management and Economics, Technical University of Denmark, Kongens Lyngby, 2800, Denmark}

\date{}

\begin{document}

\maketitle

\begin{abstract}
    A major step in the planning process of passenger railway operators is the assignment of rolling stock, i.e., train units, to the trips of the timetable.
    A wide variety of mathematical optimization models have been proposed to support this task, which we discuss and argue to be justified in order to deal with operational differences between railway operators, and hence different planning requirements, in the best possible way.  
    Our investigation focuses on two commonly used models,  the \CompositionModel{} and the \HypergraphModel{}, that were developed for Netherlands Railways (NS) and DB Fernverkehr AG (DB), respectively.
    We compare these models in a rolling stock scheduling setting similar to that of NS, which we show to be strongly NP-hard, and propose different variants of the \HypergraphModel{} to tune the model to the NS setting.
    We prove that, in this setting, the linear programming bounds of both models are equally strong as long as a  \HypergraphModel{} variant is chosen that is sufficiently expressive.
    However, through a numerical evaluation on NS instances, we show that the \CompositionModel{} is generally more compact in practice and can find optimal solutions in the shortest running time.
\end{abstract}


\section{Introduction}

Decision support tools based on optimization methods have proven their value throughout different stages in the planning process of passenger railway operators.
One of the problems that have benefited substantially from mathematical optimization methods is that of rolling stock scheduling, in which rolling stock is assigned to the trips in the timetable.
Optimization methods have shown to reduce the cost of rolling stock schedules for passenger railway operators, while at the same time increasing passenger satisfaction by offering a better match of supply and passenger demand. 
Indeed, such improvements have been documented for the rolling stock operations at both NS and DB, and they have been presented in the finals of two INFORMS Edelman competitions, see \citet{kroon2009new} and \citet{borndoerfer2021deutsche}. 

The literature offers a diverse set of models for rolling stock scheduling.
To a large extent, this diversity of models is due to the differences between the studied railway companies, both in terms of the available rolling stock and the operational constraints. 
Examples include differences in the planning horizon, being cyclical or non-cyclical, the extent to which maintenance of train units needs to be taken into account in the planning stage, and the feasibility or infeasibility of deadhead trips.
As a result, different concepts and ideas have been used to model these problems, and, likewise, different solution techniques were applied, such as the use of a commercial mixed-integer linear programming (MILP) solver \citep{fioole2006rolling}, column generation \citep{BorndoerferReutherSchlechteetal.2016,lusby2017branch}, Lagrangian relaxation \citep{cacchiani2013lagrangian} and heuristics \citep{cacchiani2019effective, hoogervorst2019variable}.

The literature on rolling stock scheduling can be categorized into several streams.
Each of these streams focuses on a particular rolling stock scheduling setting, and this setting is often represented by a common base model.
An example is the stream of papers focusing on the setting of NS, see, e.g., \citet{wagenaar2017maintenance} and \citet{kroon2015rescheduling}, which build on the model of \citet{fioole2006rolling}.
Similarly, multiple papers focus on the Train Unit Assignment Problem (TUAP) as proposed by \citet{cacchiani2010solving} and the Hypergraph model setting as introduced by \citet{BorndoerferReutherSchlechteetal.2016}.

While many of the core ideas in these streams are similar, subtle or not-so-subtle differences in the problem setting often lead to substantial formulation differences and even the need for applying different solution methods.
Given a (new) rolling stock scheduling problem, it is then often unclear which of the existing models and solution approaches is the most suitable one. 
In fact, little is known about the relative performance of different models for specific scenarios. 
A notable exception is the comparison by \citet{haahr2016comparison} between the models of \cite{fioole2006rolling} and \citet{lusby2017branch} in a setting that includes instances of NS and the Copenhagen Suburban Railway Operator DSB S-tog.
However, this comparison was based solely on numerical results and not on analytic arguments.

In this paper, we survey the rolling stock scheduling models of the literature and give insight into their relative performance by comparing two commonly used flow-based rolling stock scheduling models: the \CompositionModel{} of \citet{fioole2006rolling} and the \HypergraphModel{} of \citet{BorndoerferReutherSchlechteetal.2016}.
Our contributions in this paper are fivefold.
First, we categorize the models that have been proposed in the literature for rolling stock scheduling based on the operational context that they consider.
Second, we analytically compare the linear programming bounds provided by the \CompositionModel{} and the \HypergraphModel{}. 
It turns out that the \CompositionModel{} does not match the bounds provided by the \HypergraphModel{} in general, but it does so very well for the problems for which it was developed, namely, in a rolling stock context similar to the one at NS. 
Third, we propose for the comparison several new variants of the \HypergraphModel{} that implement different trade-offs between model conciseness and accuracy in an attempt to tune the model to the NS setting.
Fourth, we compare the two models through numerical experiments for instances of NS that are typically solved by the \CompositionModel{}. 
The \CompositionModel{} turns out to be best suited for these instances.
Fifth, we prove that the rolling stock scheduling problem is NP-hard in the strong sense, even in the NS setting and for a rolling stock fleet consisting of a single train unit type.   

The paper is organized as follows.
In \autoref{sec: problem definition}, we give a general description of the rolling stock scheduling problem.
In \autoref{sec: literature}, we categorize the existing rolling stock scheduling models based on several problem features.
We perform an analytical comparison of the \CompositionModel{} and the \HypergraphModel{} in \autoref{sec:models}, in which we also define some new variants of the \HypergraphModel{} and perform a complexity analysis.
We perform a numerical comparison between the \CompositionModel{} and the \HypergraphModel{} in \autoref{sec: numerical comparison}.
Finally, we conclude the paper in \autoref{sec: conclusion}.

\section{The \ProblemName{}}
\label{sec: problem definition}

In the \ProblemName{} (\ProblemAcronym{}), rolling stock is assigned to the trips in the timetable. 
Each trip indicates that a train drives from a departure station to an arrival station, with a fixed departure and arrival time. 
Some of the trips are naturally joined together into timetable services, which define a train service between two terminal stations of a railway line.
At these terminal stations, rolling stock then moves from one timetable service to another through so-called turnings.

The rolling stock that is available to operate the trips differs per railway operator. 
In the \ProblemAcronym{}, we assume rolling stock that is composed of self-propelled bi-directional train units, so without a separate locomotive, as is the case for many European passenger railway operators.
These train units can, generally, be coupled together to form \textit{compositions}, which differentiates rolling stock scheduling from, e.g., vehicle or aircraft scheduling.
Combining train units allows for more passenger transportation capacity on a trip.
Moreover, it generally requires less energy, crew members, and infrastructure capacity to run a train of multiple coupled units than running each of the units individually.

The compositions of train units are generally not fixed throughout the day, but can be changed at some of the stations at which a train stops.
This allows adjusting the offered capacity to varying passenger numbers over the day.
Composition changes may happen at either some in-between station or a terminal station of the timetable service.
Adding train units to a composition is often referred to as \textit{coupling} of train units, while removing train units is referred to as \textit{uncoupling} of train units.
Jointly, such coupling and uncoupling actions are referred to as \textit{shunting} of train units.
When train units are uncoupled from a composition, they are often parked at a shunting yard, in which case we speak about the \textit{pulling in} and \textit{pulling out} of train units to or from the shunting yard or depot.

When determining how a train unit moves from one trip to another, we need to satisfy business rules that are dependent on the operator.
In some cases, trips have an assigned successor trip and train units can only move from the predecessor trip to the successor trip, unless the train unit is uncoupled from the composition.
This is especially the case when two trips are part of the same timetable service.
In other cases, train units can move between any trips as long as this is possible with respect to the  time and the location of the trips.
In particular, \textit{deadheading}, i.e., running the train without any passengers, may be allowed to move from the arrival station of one trip to the departure station of the other.

Additional conditions may need to be satisfied when determining the trips that a train unit can operate during the planning horizon.
For example, maintenance restrictions often need to be satisfied, such as a maximum distance between two consecutive maintenance appointments.
In addition, restrictions generally apply to the compositions that are allowed for trips and the ways in which these compositions can be changed.
The length of a composition can, for example, be limited by the platform length of the stations that are passed, while coupling and uncoupling can often only happen at one side of the composition due to the station layout.

In the \ProblemAcronym{}, we are now asked to find an assignment of the available train units to the trips in the timetable.
The path for each train unit that is implied by this assignment should satisfy all the above-mentioned constraints.
In addition, the composition that is chosen for each trip in this assignment should be feasible, as well as the ways in which the compositions are changed at stations.
The result of the \ProblemAcronym{} is referred to as a rotation, or circulation, for the rolling stock units.

An example of a rolling stock circulation as can be encountered at NS is shown in \autoref{fig: circulation example}, which shows an optimal solution of instance DDZ of the test set to be discussed in \autoref{sec: numerical comparison}.
Each colored node on the left-hand side of this figure corresponds to a train unit that is available at the start of the day, the color indicating the type of this train unit.
Similarly, the nodes on the right-hand side indicate train units ending the day.
All in-between nodes represent the operation of a trip by train unit(s), where the size of the node stands for the number of train units that are part of the composition that operates the trip.
The arcs between the train units show how train units move from trip to trip, such that a path corresponds to the trips operated by a single train unit over the day.

A few things can be noted in the instance that are typical of instances of NS.
First, most trips have a successor trip on which all train units of the composition continue, meaning that shunting only happens at a limited number of places.
Moreover, many trips in the morning peak are operated by double train unit compositions, followed by trips operated by single train units around noon with lower demand.
The long arcs that cross this period represent parking at some depot, from which the units are pulled out later to service the afternoon peak with many double train unit compositions again.
After the afternoon peak, a lot of the train units are uncoupled again and either parked at the shunting yard overnight or used within a single train unit composition during this period with lower demand.
Clearly, the crucial point in such an instance is to get the shifting of units right in order to service the demand in the most efficient way.

\begin{figure}[htbp]
    \centering
    \includegraphics[width=\textwidth, height=0.4\textheight]
      {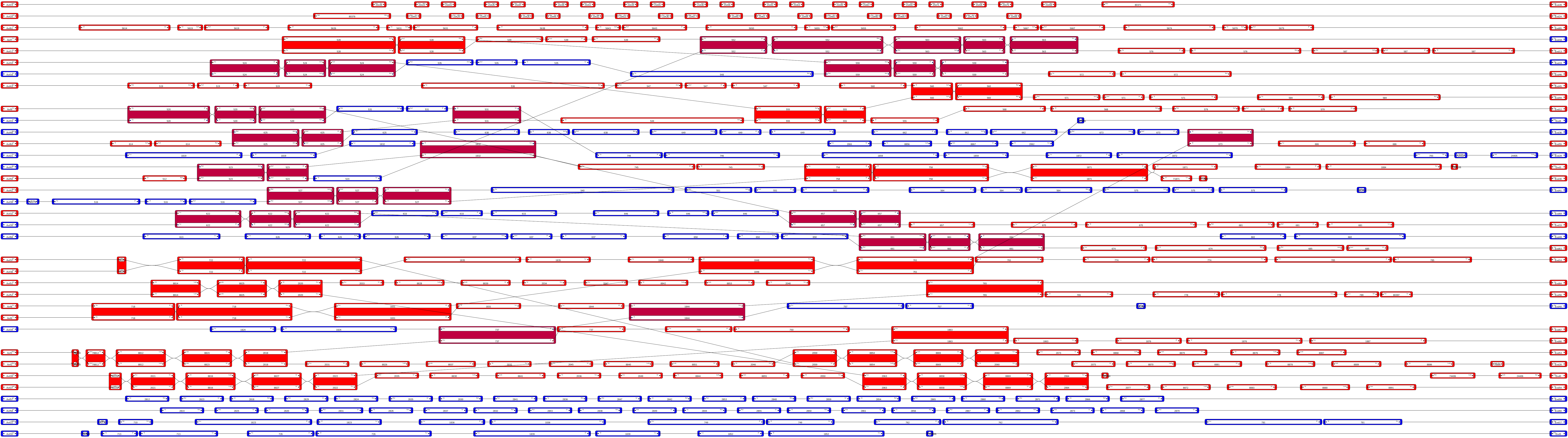}
	\caption{Example of a rolling stock circulation.}
	\label{fig: circulation example}
\end{figure}

\section{Feature-Driven Model Comparison}
\label{sec: literature}

The planning process of a railway passenger operator is generally composed of a number of planning steps that are carried out sequentially.
In each of these steps, a single operational problem is considered and the solution of this problem is the input to the next planning step.
An overview of the typical planning steps is given by, among others, \citet{huisman2005operations} and \cite{caprara2007passenger}.
In both of these papers, rolling stock scheduling is assumed to be performed after a timetable has been found and before the crew, i.e., drivers and guards, are scheduled.
In this paper, we thus focus solely on the rolling stock scheduling step.

Over the years, multiple models have been proposed for the rolling stock scheduling problem.
To a large extent, the differences between these models are the result of differences in the operational contexts of the various passenger railway operators.
These arise, e.g., from differences in the available rolling stock fleet, as well as from differences in the national (infrastructural) context.
In the remainder of this section, we categorize the available models according to five main features of the operational context that these models consider.
Moreover, we identify streams of papers in case multiple papers share the same core model.
In this way, we complement the model classifications that were provided earlier by \citet{thorlacius2015integrated} and \citet{schlechte2023bouquet}.
We limit our categorization to models for rolling stock scheduling that focus on a passenger, as opposed to freight, railway context and in which train units, as opposed to locomotives and carriages, are considered.
Moreover, we focus our categorization on papers considering rolling stock scheduling within a network, as opposed to a single-line, context.
A schematic overview of our categorization is given in \autoref{tab: literature categorization}.

\begin{sidewaystable}[htbp!]
    \definecolor{Gray}{gray}{0.94}
    \newcolumntype{a}{>{\columncolor{Gray}}c}
    
    \centering
    \resizebox{\textwidth}{!}{
    \begin{tabular}{lcacacacacacaccccc}
        \toprule
        Paper &  \multicolumn{3}{c}{Compos.} & \multicolumn{3}{c}{Turnings} & \multicolumn{2}{c}{Mainten.} & \multicolumn{3}{c}{Passengers} & \multicolumn{2}{c}{Other} & \\
        & \rotatebox[origin=c]{270}{Number} & \rotatebox[origin=c]{270}{Ordering} & \rotatebox[origin=c]{270}{Orientation} & \rotatebox[origin=c]{270}{Fixed} & \rotatebox[origin=c]{270}{Flexible} &
         \rotatebox[origin=c]{270}{Deadheading} &
         \rotatebox[origin=c]{270}{Distance} & \rotatebox[origin=c]{270}{Appointment} &
        \rotatebox[origin=c]{270}{Constraint} &
        \rotatebox[origin=c]{270}{Objective} &
        \rotatebox[origin=c]{270}{Dynamic} & \rotatebox[origin=c]{270}{Robustness} & \rotatebox[origin=c]{270}{Regularity} & Core Model & Sol. Method & Company\\
        \midrule
        \citet{fioole2006rolling} & & \checkmark & & \checkmark & & & & & & \checkmark & & & & Composition Model & MILP solver & NS \\
        \citet{peeters2008circulation} & & \checkmark & & \checkmark & & & & & & \checkmark & & & & & B\&P & NS \\
        \citet{cacchiani2010solving}  & \checkmark & & & & \checkmark & \checkmark & (\checkmark) & & \checkmark & & & & & TUAP Model & Col. Gen. Heur. & Regional \\
        \citet{cadarso2011robust} & \checkmark & & & & \checkmark & \checkmark & & & & \checkmark & & (\checkmark) & & \citeauthor{cadarso2011robust} Model & MILP solver & RENFE \\
        \citet{cacchiani2012railway} & & \checkmark & & \checkmark & & & & & & \checkmark & & \checkmark & & Composition Model & Benders Heur. & NS \\
        \citet{nielsen2012rolling} & & \checkmark & & \checkmark & & & & & & \checkmark & & & & Composition Model & MILP solver & NS \\
        \citet{cacchiani2013lagrangian} & \checkmark & & & & \checkmark & \checkmark & & & \checkmark & & & & & TUAP Model & Lagr. heuristic & Regional \\
        \citet{cadarso2014improving} & & \checkmark & & & \checkmark & \checkmark & & & & \checkmark & & (\checkmark) & & \citeauthor{cadarso2011robust} Model & Benders Decomp. & RENFE \\
        \citet{giacco2014rolling} & & & & & \checkmark & \checkmark & \checkmark & & & & & & & & MILP solver & Trenitalia \\
         \citet{lin2014two} & & \checkmark & & & \checkmark & \checkmark & & & \checkmark & & & & & \citeauthor{lin2014two} Model & 2-stage approach & Southern \\
        \citet{kroon2015rescheduling} & & \checkmark & & \checkmark & & & & & & & \checkmark & & & Composition Model & MILP solver & NS \\
        \citet{thorlacius2015integrated} & & \checkmark & & & \checkmark & \checkmark & \checkmark & & \checkmark & & & & & & Heuristic & DSB S-Tog \\ 
        \cite{BorndoerferReutherSchlechteetal.2016} & & & \checkmark & & \checkmark & \checkmark & \checkmark & & \checkmark & & & & \checkmark & Hypergraph Model & B\&P & DB\\
        \citet{lin2016branch} & \checkmark & & & & \checkmark & \checkmark & & & \checkmark & & & & & \citeauthor{lin2014two} Model & B\&P & Scotrail  \\
        \cite{Grimm2017Acyclic} & & & \checkmark & & \checkmark & \checkmark & \checkmark & & & & & & \checkmark & Hypergraph Model & B\&P & DB \\
        \citet{lusby2017branch} & \checkmark & & & \checkmark & & & (\checkmark) & & & \checkmark & & & & & B\&P & DSB S-tog \\
        \citet{wagenaar2017rolling} & & \checkmark & & \checkmark & & \checkmark & & & & \checkmark & & & & Composition Model & MILP solver & NS\\
        \citet{wagenaar2017maintenance} & & \checkmark & & \checkmark & & & & \checkmark & & \checkmark & & & & Composition Model & MILP solver & NS \\
        \citet{cacchiani2019effective} & \checkmark & & & & \checkmark & \checkmark & & & \checkmark & & & & & & Heuristic & Regional  \\
        \cite{GrimmBorndoerferReutheretal.2019} & & & \checkmark & & \checkmark & \checkmark & \checkmark & & & & & & \checkmark & Hypergraph Model & B\&P & DB \\
        \citet{zhong2019high} & & \checkmark & & \checkmark & & & \checkmark & & \checkmark & & & & & Composition Model & Matheuristic & CHSR  \\
        \citet{gao2020branch} & \checkmark & & & & \checkmark & (\checkmark) & \checkmark & & \checkmark & & & & \checkmark & & B\&P & CHSR \\
        \citet{gao2022weekly} & \checkmark & & & & \checkmark & (\checkmark) & \checkmark & & \checkmark & & & & & & B\&P & CHSR \\
        \cite{GrimmBorndoerferBushe2023} & & & \checkmark & & \checkmark & \checkmark & \checkmark & & & & & & \checkmark & Hypergraph Model & B\&P & DB \\
        \bottomrule 
    \end{tabular}}
    \caption{Comparison of the models in the literature based on different characteristics of the rolling stock scheduling problem: the extent to which details of the compositions are considered, the flexibility available in turnings between trips, the way in which maintenance is taken into account, the way in which passenger capacity is considered, and a few other characteristics. Parentheses around a check mark symbol indicate that a paper partially fulfills a certain requirement.}
    \label{tab: literature categorization}
\end{sidewaystable}

\subsection{Compositions}
A first difference between the models concerns the level of detail that is considered in modeling the compositions into which train units can be coupled.
On the one hand, \citet{cacchiani2010solving}, \citet{lusby2017branch}, \citet{lin2016branch}, and \citet{cadarso2011robust} consider only the number of train units of each type in a composition, but not the order of train units within a composition.
Such models are mainly used in a setting where there is only a single train unit type or when the compositions are short and where uncoupling and coupling can happen on both sides of the composition.

Models taking into account the exact order of the train units within a composition include those of \citet{fioole2006rolling} and \citet{peeters2008circulation}.
Note that these models do not track the individual train units but do track what type of train unit is in each position of the composition.
Taking into account the order of train units is important in these operational contexts, as coupling and uncoupling of train units is often only allowed on one side of the composition due to the station layout.
Hence, it is necessary to track the order of train units in the composition to determine if a certain train unit can be coupled or uncoupled.
A similar restriction is taken into account by \citet{lin2014two}, who determine the order of the train units in a composition in a second optimization step after determining in a first optimization step how many train units of each type are in a composition.

Finally, \citet{BorndoerferReutherSchlechteetal.2016} do not only take the order of train units in a composition into account, but also the orientation of the train units themselves, i.e., which side of each train unit is facing in the direction of travel.
This is done as infrastructure requirements in the setting of DB prevent some orientations of the train units in a composition and as the orientation of the individual train units is of importance for the seat reservations due to 1st and 2nd class being on opposite sides of train units.
Note that the models of \citet{fioole2006rolling} and \citet{peeters2008circulation} extend canonically to take orientations into account by creating a separate composition for each possible combination of orientations of the train units.

\subsection{Rolling Stock Turnings}
A second distinction between the models is in the way that turnings between trips are handled.
On the one hand, these turnings can be determined in advance.
In that case, each trip has a fixed follow-on trip, unless all train units of the trip are parked at the station after the trip. 
All train units that are part of the composition on the predecessor trip then move to the fixed successor trip or are uncoupled from the composition and can only be coupled to another trip after some fixed reallocation time.
Such a setting is considered, e.g., by \citet{fioole2006rolling} and \citet{lusby2017branch}.
Fixing the turnings is especially common in high-density networks, where there are limited possibilities to execute shunting at the stations.

Alternatively, determining the rolling stock turnings can be part of the decisions that are made in the model.
Examples include the models of \citet{cacchiani2010solving}, \citet{giacco2014rolling}, \citet{lin2014two}, and \citet{BorndoerferReutherSchlechteetal.2016}.
For these models, possible turnings are determined based on business rules that state which turnings are feasible and acceptable for the operators.
The model can then determine which of these turnings to take.
Note that in such settings each train unit in the composition can often follow its own turning, meaning that a composition is split into smaller parts that each turn to a different trip.
The advantage of allowing for flexible turning possibilities is that more efficient rolling stock circulations can be found, although often at the expense of more irregular turning patterns at the stations.

Another difference, which is strongly related to the understanding of turnings within the models, is the inclusion of deadheading trips.
In some models, such as those of \citet{fioole2006rolling} and \citet{lusby2017branch}, deadheading is not directly considered.
Instead, deadheading can only occur in those models in case it has been planned beforehand.
Models that do allow to assign deadheading within the model include those of \citet{giacco2014rolling}, \citet{cadarso2011robust}, \citet{BorndoerferReutherSchlechteetal.2016}, and \citet{wagenaar2017rolling}.
In these models, deadheading can be used either to move between trips in a turning or to move to a maintenance location.
As can be seen in \autoref{tab: literature categorization}, such flexibility in assigning deadheading trips is strongly linked to the flexibility in turning in the models, where all but one of the models that allow for flexibility in deadheading also do so in the turnings.

\subsection{Maintenance Requirements}
An important restriction when scheduling the rolling stock at many railway operators is that maintenance requirements need to be satisfied.
Such requirements ensure that enough maintenance is performed to adhere to safety and quality standards.
Some papers, like those of \citet{fioole2006rolling}, \citet{nielsen2012rolling}, \citet{cadarso2011robust}, and \citet{cacchiani2012railway}, do not consider any maintenance requirements.
In such models, it is often assumed, either implicitly or explicitly, that maintenance can be scheduled after a circulation has been found.
This is, for example, the case when maintenance can be executed during the evening hours when only few trains are operated.
Alternatively, maintenance may be scheduled closer to the day of operation during the time that a train unit is parked at the shunting yard of a station \citep[see, e.g.,][]{maroti2005maintenance}.

Other papers have considered distance or time-based maintenance requirements, where train units need to undergo maintenance after a certain amount of use.
These include the models of \citet{cacchiani2010solving}, \cite{BorndoerferReutherSchlechteetal.2016}, \citet{lusby2017branch}, \cite{GrimmBorndoerferReutheretal.2019}, and \cite{GrimmBorndoerferBushe2023}.
The approach taken to include such requirements differs between these models.
While \citet{lusby2017branch} and \cite{GrimmBorndoerferBushe2023} consider an individual path for each train unit through the network, \citet{BorndoerferReutherSchlechteetal.2016} consider an additional resource flow to keep track of the distance traveled by each train unit. 
An almost identical approach to the latter one is chosen in \citet{giacco2014rolling}, who investigate the benefit of additional deadheading options to reach maintenance facilities in order to decrease the total maintenance hours of a rotation.
A different point of view is taken by \citet{cacchiani2010solving}, who require that a certain fraction of the found train unit paths through the network allow for a maintenance activity. 
The approach of \cite{GrimmBorndoerferReutheretal.2019} is instead based an a cutting plane algorithm that cuts of rotations that violate maintenance constraints during the solution process.

A further type of maintenance requirement is considered by \citet{wagenaar2017maintenance}, where some of the train units have a fixed maintenance appointment.
Such an appointment specifies both the location and time at which the maintenance takes place and is usually planned close to the day of operation.
\citet{wagenaar2017maintenance} propose and compare three different models for rolling stock rescheduling to ensure that train units make these appointments.

\subsection{Passenger Capacity}
Another difference between the models is the way in which the passenger demand is taken into account.
On the one hand, passenger demand can be enforced through a constraint on the required passenger capacity, i.e., a minimum number of seats per trip.
This is done by, e.g., \citet{cacchiani2010solving}, \citet{BorndoerferReutherSchlechteetal.2016}, \citet{lin2016branch}, and \citet{thorlacius2015integrated}.
In these models, the main focus is then on minimizing the costs of the railway operator given these capacity constraints.
However, many of these models can easily be extended to also include passenger demand as an objective.

Alternatively, some models deal with passenger demand by also considering shortages of capacity.
Examples include the models of \citet{fioole2006rolling}, \citet{cadarso2011robust}, and \citet{lusby2017branch}, where any shortages of seats compared to the expected passenger demands are penalized in the objective function.
Note that these papers use fixed passenger demands, where the flow of passengers does not change based on the chosen rolling stock assignment.
\citet{kroon2015rescheduling} instead consider flexible passenger flows for a rescheduling setting in which the capacity of some trains may not be sufficient to accommodate all passengers.
As the interaction between rolling stock assignment and passenger flows is hard to handle in an integrated MILP model, these authors consider an iterative framework to solve this problem that alternates between rolling stock rescheduling and the rerouting of passengers.

\subsection{Other Model Differences}
The characterization above is certainly not complete concerning the differences between the models.
In \autoref{tab: literature categorization}, we state a few other differences.

The first is the inclusion of some form of robustness, which ensures that a good rolling stock circulation is found even when some of the details regarding the rolling stock scheduling problem are uncertain.
For example, \citet{cacchiani2012railway} propose a robust two-stage optimization model to better deal with large disruptions.
In this way, a rolling stock circulation is found that can be rescheduled well when such a large disruption occurs.
Also \citet{cadarso2014improving} look at the robustness of solutions, where they focus on creating a circulation that is likely to be robust in execution.
They do this by, e.g., penalizing shunting operations that are likely difficult to execute and by penalizing the expected delay that follows from a certain shunting action.

Another difference concerns the inclusion of regularity, i.e., the uniformity of the rolling stock plan over different days in the planning horizon; this is particularly relevant in case the timetable is (largely) periodic.
\citet{Borndoerfer2017Reoptimization}, e.g., focus on the regularity of turnings by penalizing in the objective function cases in which a different turning is chosen between the same trips on different days of the planning horizon.
\citet{gao2020branch} instead include a hard constraint, which requires that a trip that is operated on both days of a two-day planning horizon is executed by the same train sequence.
For a railway operator, inclusion of such regularity requirements, e.g., allows for a more compact representation of the circulation and makes rescheduling over multiple days easier.

\subsection{Discussion}
The categorization in \autoref{tab: literature categorization} shows that a number of streams can be identified in the literature, within which the papers use a similar base model and focus on a similar problem setting.
For example, numerous papers use the \CompositionModel{} as introduced in \citet{fioole2006rolling} and share a problem setting that considers the ordering of train units in compositions, fixed turnings between trips and the inclusion of passenger demand in the objective.
Other common streams include the papers using the \HypergraphModel{} introduced in \citet{BorndoerferReutherSchlechteetal.2016}, sharing a problem setting in which the orientation of train units in a composition and flexible turnings are considered, and those using the model for the TUAP introduced in \citet{cacchiani2010solving}, sharing a problem setting that focuses on the number of train units of each type in a composition and has flexible turnings.

The presence of these streams in the literature shows that the developed models have been highly adapted to their problem setting.
However, similarities can also be seen between the streams.
For example, the \HypergraphModel{} stream generalizes the properties of the \CompositionModel{} and TUAP model on many of the problem characteristics, e.g., considering the most detailed tracking of compositions.
In addition, variations have been considered of the models which bridge some of the gaps between the different streams.
An example includes the inclusion of deadheading in the \CompositionModel{} in \citet{wagenaar2017rolling}, bringing it closer to models that do consider deadheading.
A natural question would therefore be how these models relate to each other analytically.

To the best knowledge of the authors, little is generally known on the relative performance of the models and their theoretical applicability to problem settings that are different than the one for which they were proposed.
The only exception includes the numerical comparison made by \citet{haahr2016comparison} between the \CompositionModel{} and model proposed by \citet{lusby2017branch}.
Therefore, we focus in this paper on gaining insight on the performance of two models that can be associated to a stream in the literature: the \CompositionModel{} and the \HypergraphModel{}.
Such a comparison can both give theoretical insights for researchers working on rolling stock scheduling, as well as help practitioners choose an appropriate model to solve rolling stock scheduling problems.

\section{Analytic Model Comparison}
\label{sec:models}

In this section, we analytically relate the \CompositionModel as proposed by \cite{fioole2006rolling} and the \HypergraphModel as proposed by \cite{BorndoerferReutherSchlechteetal.2016}.
The notation used in our comparison is listed in Table~\ref{tab:symbols}.
Our analysis will revolve around the problem characteristics that we discussed in Section~\ref{sec: literature}, and one of our goals is to identify the relevant factors and their influence. This is somewhat challenging, as both approaches are general in the sense that they can be, and have been, adapted to different scenarios as discussed in \autoref{sec: literature}, such that it is not so clear what ``the \CompositionModel{}'' or, to a broader extent, ``the \HypergraphModel{}'' actually is. We therefore discuss several model variants. But whatever variant is considered, the models have been developed for specific settings, in which, inevitably, their particular strengths and weaknesses surface. Indeed, we will argue that the \CompositionModel is preferable to the \HypergraphModel in a setting like the one at NS, in which potential advantages of the \HypergraphModel{} do not materialize. Our point is that there is no ``best model'' that suits all rolling stock scheduling problems, it is, rather to the contrary, important to choose the right model for a particular application. 

\begin{table}[ht]
    \centering
    \footnotesize
    \begin{tabular}{ll}
    	 \toprule
         Symbol     & Meaning \\
         \midrule
         $t\in T$   & timetabled trip \\
         $c\in C$   & connection between two trips \\
         $p\in P$   & train composition \\
         $t^+, t^-$ & departure, arrival (event) associated with trip $t$ \\
         $n\in [n_{\max}]$ & position in train composition \\
         $r\in R$   & train unit type \\
         $d\in D$   & depot (to park train units) \\
         $G=(V,H)$  & hypergraph with event-activity nodes $V$ and hyperarcs $H$; we write $H(G)=H$ etc.\\ 
         $G'=(V,B)$ & base graph with event-activity nodes $V$ and train unit activity arcs $B$ \\
         $H_1, H_{\geq 2}$ & hyperarcs consisting of one or more than one base arc \\
         $A=H_{1,D}$ & set of all pull-in/out, parking, and direct connection arcs w.r.t.\ a model \\ 
         $b_v$      & train unit balance at node $v$ \\
         $\hD, \hA, \HD, \HA, \C$ & model variants: $\rm h/H/C$ = small/full/composition hypergraph, $\rm D/A$ = depot/direct connections \\
         $G_{\rm M}$      & hypergraph variants,  $\rm M\in\{\hD, \HD, \hA, \HA, \C\}$ \\
         $p_a$      & path in $G_\hD, G_\HD$ represented by direct connection arc $a$ in $G_\hA, G_\HA$ \\
         $\olA$    & closure of (all possible) direct connection arcs \\ 
         $G_\holA, G_\HolA$ & hypergraph variants with all possible direct connection arcs  \\
         $x,x'$     & hyperflow in $G$ and corresponding base flow in $G'$ \\
         $V^{\leq v}_D$ & set of depot nodes $V_{D,r(v)}$ of type $r(v)$ up to and including node $v\in V_D$ (w.r.t.\ time) \\ 
         $0_v$      & first node in $V^{\leq v}_D$ for depot node $v\in V_D$ (inventory start) \\
         $H^+_v, H^-_v$ & set of composition change hyperarcs that imply a pull-out/in 
                          of a train unit from depot node $v$ \\
         $\nu^r_h$  & number of train units of type $r$ pulled out/pulled in because of 
                      hyperarc $h\in H^+\cup H^-$ \\
         $\IPM$   & IP model associated with model $\rm M$, $\rm M\in\{\hD, \HD, \hA, \HA, \C\}$\\
         $\PIP(M), \PLP(M)$  & integer and fractional polytope associated with model $\IPM$, $\rm M\in\{\hD, \HD, \hA, \HA, \C\}$ \\ 
         $\nu_\IP(M), \nu_{\LP}(M)$ & optimal objective/LP relaxation value associated with model $\IPM$, $\rm M\in\{\hD, \hA, \HD, \HA, \C\}$ \\
         \bottomrule
    \end{tabular}
    \normalsize
    \caption{List of symbols used in the analytic comparison. We use indices to denote subsets, e.g., $P_t\subseteq P$ is the set of compositions for train $t\in T$.}
    \label{tab:symbols}
\end{table}

\subsection{Flows in Graph-Based Hypergraphs}

Rolling stock scheduling problems can be modeled in terms of hyperflows. To this purpose, we consider graph-based directed hypergraphs $G=(V,H)$. They consist of nodes $V$ and hyperarcs $H$, where each hyperarc is a union of disjoint arcs from an underlying directed base graph $G'=(V,B)$ on the same set of nodes.
Let us denote by $H_1 \subseteq H$ the hyperarcs that consist of a single arc (and are hence standard arcs), and by $H_{\geq 2}=H\setminus H_1$ all genuine hyperarcs. If each arc in the base graph models a movement of an individual train unit, a path a sequence of such movements, and a flow the movement of a fleet of train units, then a hyperarc in an associated graph-based hypergraph models the joint movement of several train units in a train composition, and a hyperflow the movement of a fleet in compositions.
A hyperflow is thus a vector $x\in \R_{\geq 0}^H$ that satisfies flow conservation constraints $x(\delta^+(v))-x(\delta^-(v))=b_v$ at every node $v$ subject to node balances $b_v$; here and elsewhere we write $x(H')=\sum_{h\in H'}x_h$ for $H'\subseteq H$.
These node balances will be mostly 0, but can also model surplus and deficit train unit inventories. As each hyperarc is a union of standard arcs, a hyperflow $x$ in $G$ defines a flow $x'$ in $G'$ by breaking the hyperarcs into arcs as 
\[
  x'_a:=\sum_{h\ni a} x_h,  \qquad \forall a\in A
\]
w.r.t.\ the same node balances.
Hence, every hyperflow in a graph-based hypergraph decomposes into paths (in the base graph), just like a standard flow.   

\subsection{The NS Setting: Basic Assumptions}

To model a rolling stock scheduling problem in terms of graph-based hyperflows, one introduces nodes that identify the types, positions, and orientations of the train units in a train composition of a particular trip, and hyperarcs that describe the movements and turns of the train units in these compositions, connecting appropriate pairs of nodes. 
One possible choice (to be discussed later) are nodes of the form $v=(t^\pm,r,n)$, where $t\in T$ denotes the trip, $+$ or $-$ departure or arrival, $r\in R$ the train unit type, and $n\in [n_{\max}]$ the train unit position in the composition.
Here, $n_{\max}$ is the maximum number of train units in a composition.
Some companies like DB also use an orientation $o\in\{\tick,\tock\}$, where $\tick$ denotes a forward and $\tock$ a backward orientation w.r.t.\ the driving direction. 
Trip hyperarcs can then connect appropriate departure and arrival nodes of a trip, and turn hyperarcs can connect arrival and departure nodes of a trip and a follow-on trip. 

How this general idea is implemented depends on the scenario at hand. 
We consider in this paper an  ``NS setting'', that restricts train unit movements on connections between timetabled trips in a way that is tailored to the situation at Netherlands Railways. Namely, an NS setting has the following basic properties:
\begin{itemize}
\item[i)] Deadheading can only occur when an explicit ``dummy trip'' has been added to the timetable beforehand.
\item[ii)] Train units can (only) be coupled to or uncoupled from trains via depots at stations, i.e., train units that are uncoupled move into the depot, and train units that are coupled to  a compostion come out of the depot.
\item[iii)] Most trips have exactly one follow-on, and exactly one predecessor trip; a pair of such trips forms a so-called 1-to-1 connection. 
\item[iv)] There are (a few) trips whose composition can be split into at most two parts, or joined from at most two parts; such trips then have at most one predecessor/successor trip, respectively.
The first-mentioned trips form what we call a 1-to-2 or a 2-to-1 connection, and it is not allowed to join a split, or to split a join, i.e., 1-to-2 and 2-to-1 connections must be disjoint. 
\end{itemize}
The above properties lead to a set of defined connections between trips, which we will denote by $C$.

The NS setting aims at a service that operates several lines (at high frequencies) that give rise to sequences of follow-on trips. Train units are coupled to and uncoupled from the trains at relatively many stations to adjust for capacity needs, parking them at the stations in sidings for later reuse. At a few places, trains are split or joined. If deadhead trips are necessary, they are added manually as ``dummy trips''. 

\subsection{Modelling Rolling Stock Rotations by Hyperflows}

A typical turn in an NS setting is a connection that involves some train units that continue with the follow-on trip(s), some that pull-in to the station, and some that pull-out, usually, but not necessarily, only one of the latter two. 
These activities belong together and can be modeled in terms of a single hyperarc. 
A more compact alternative that we will use in this paper is to break off the pull-in and pull-out activities and model them by separate arcs. 
To this purpose, one sets up a station timeline for each train unit type, with a node for each possible arrival of a pull-in activity and each possible departure of a pull-out activity.
We then add pull-in arcs that connect appropriate arrival nodes to the timeline, pull-out arcs that connect the timeline to appropriate departure nodes, and parking arcs that connect timeline nodes that are subsequent in time, including an initial and a terminal node that represent the start and end inventory. 
We call such a timeline a depot, and there is one depot for each train unit type at every station. We denote the set of all depots by $D$. Note that the train unit flow on parking arcs must be allowed to take general non-negative integer values, which makes train unit tracking impossible.
We remark, though, that train unit tracking can be handled by introducing separate parking tracks for individual train units, but this blows up the model and introduces degeneracy, and we will not consider this option in this paper.  

\begin{figure}[htb] 
  \centering
  \begin{subfigure}[b]{0.49\textwidth}
    \centering
    \includegraphics[scale=1.4]{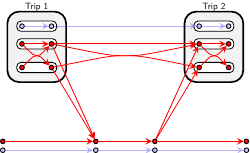}
    \vspace{0.58cm}
    \caption{$G_\hD$}
    \label{fig:hD}
  \end{subfigure}
  \begin{subfigure}[b]{0.49\textwidth}
    \centering
    \includegraphics[scale=1.4]{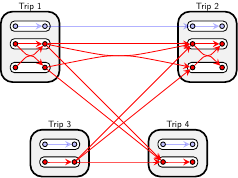}
    \caption{$G_\hA$}
    \label{fig:hA}
  \end{subfigure}
  \begin{subfigure}[b]{0.49\textwidth}
    \centering
    \includegraphics[scale=1.4]{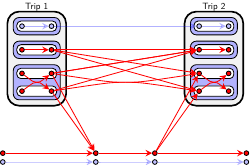}
    \vspace{0.44cm}
    \caption{$G_\HD$}
    \label{fig:HDM}
  \end{subfigure}
  \begin{subfigure}[b]{0.49\textwidth}
    \centering
    \includegraphics[scale=1.4]{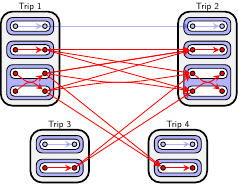}
    \caption{$G_\HA$}
    \label{fig:HAM}
  \end{subfigure}
  
  \caption{Illustration of hypergraph variants: h = small hypergraph, H = full hypergraph, D = depot connections, A = direct connections.}
  \label{fig:small_Hypergraphs}
\end{figure}

The model resulting from the above construction is illustrated in \autoref{fig:hD} and denoted $G_\hD$.
Here, $\rm h$ stands for a ``small hypergraph'', and $D$ for train unit transfers via a depot. 
There are two trips (Trip 1 and 2), which can be serviced by a single train unit composition of a blue train unit, a single train unit composition of a red train unit, or a double train unit composition of two red train units.
These options are modeled by blue standard arcs, red standard arcs, and red hyperarcs, respectively. 
Trip 1 is connected to Trip 2 by arcs and hyperarcs that model straight continuations of these compositions. There is also the option to operate Trip 1 with a red double train unit composition, and to continue with a red single train unit composition. 
In this case, the red train unit at (the lower) position 2 pulls out via a standard arc to be parked at the depot. 
Or the other way round, we arrive with a red single train unit composition, and pull a red train unit out of the red depot to continue with a red double train unit composition. 

This model $G_\hD$ has some problems to control the feasibility and cost of train unit transfers to and from a depot, namely, that one can pull-in and out from any matching depot node, which might not be feasible in practice. For instance, in the example in \autoref{fig:hD}, we cannot forbid that the red double train unit composition is continued with the turn for the red single train unit composition and that Trip 2 is afterwards continued with the red double train unit composition, which can be achieved by using a pull-out and successive pull-in of the second train unit.
In practice, this might be an option that a railway operator wants to rule out.
If such a situation arises, one can resort to direct connections without any depot, as illustrated in \autoref{fig:hA}.
This model variant is denoted by $G_\hA$, where $A$ stands for transfers via ``direct connection arcs''.
In this model variant, depot nodes are only needed for the start and end inventories; we omit these nodes in the model snippets in our illustrations. Of course, the gain of control by direct connections increases the model size.

\begin{figure}[htbp]
\centering
\begin{subfigure}[b]{0.49\textwidth}
  \centering
  \includegraphics[page=2,scale=1.05]{./tikzpictures/IP-hAvsHA.pdf}
  \caption{Coupling Situation 1 for $G_\hA$.}
  \label{fig:ip difference hA}
\end{subfigure}
\begin{subfigure}[b]{0.49\textwidth}
  \centering
  \includegraphics[page=1]{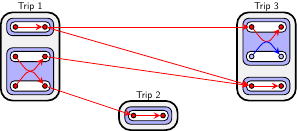}
  \caption{Coupling Situation 1$'$ for $G_\HA$.}
 \label{fig:ip difference HA}
\end{subfigure}
\begin{subfigure}[b]{0.49\textwidth}
  \centering
  \vspace*{2mm}
  \includegraphics[page=2,scale=1.3]{./tikzpictures/Lp-hAVsHA.pdf}
  \caption{Coupling Situation 2 for $G_\hD$.}
  \label{fig:strictIeqB}
\end{subfigure}
\begin{subfigure}[b]{0.49\textwidth}
  \centering
  \vspace*{2mm}
  \includegraphics[page=1,scale=1.05]{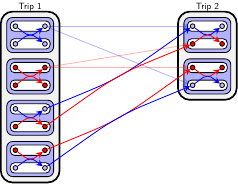}
  \caption{Coupling Situation 2 for $G_\HD$.}
  \label{fig:strictIeqA}
\end{subfigure}\caption{Coupling situations in full and small \HypergraphModel{}s.}
\label{fig:strictIeq}
\end{figure}

There is another issue with the small \HypergraphModel{}s $G_{\rm h\cdot}$ that is illustrated in \autoref{fig:ip difference hA} for the direct arc connection case $G_\hA$. In this instance of, say, Situation 1, the two train units of the pure red composition can continue in various compositions, all of which are purely red. In the small \HypergraphModel{}, however, there is an option to pull a blue train unit from the depot to continue with one red train unit in a mixed composition (whose actual predecessor composition is not shown).
In this Situation 1, the model accommodates an illegal coupling that produces an infeasible solution, i.e., the small \HypergraphModel{} is not correct.
But even if this is allowed in, say, an extended Situation 1$'$, a problem to express coupling costs in terms of arc costs comes up, because the initial and terminal composition of a composition change arc is unknown. For example, the top red arc in 
\autoref{fig:ip difference hA} represents three transitions of a red train unit between three different pairs of compositions, namely, between two red single train unit compositions, between a red single and a mixed double train unit composition, and between a red double and a red single train unit composition. Because of the first transition, this arc can bear no coupling costs at all.
If such costs accrue in the other two transitions, they must be associated with complementing pull-in and pull-out arcs, i.e., the coupling costs must be sums of costs associated with pull-in and pull-out arcs, and these costs must be the same for all their target compositions. This ``additive'' coupling cost model is restrictive and, if applied conservatively, leads to an underestimation of true coupling costs.\footnote{For two train units, one can put coupling costs on depot arcs, but this does not work for longer compositions.}
It will turn out that such situations exist in real-world scenarios; in fact, Situation 1 in \autoref{fig:ip difference hA} comes directly from our test set. 

These problems of the small \HypergraphModel{}s can be eliminated by using separate copies of nodes for each composition, i.e., by using node labels $v=(t^\pm,p,r,n)$ that are extended by a composition coordinate $p\in P_t$ for all feasible compositions $P_t$ for trip $t$. This results in two more model variants that use train unit transfers via depots and via direct connection arcs, as illustrated in \autoref{fig:HDM} and \autoref{fig:HAM}, respectively.
These models are denoted as $G_\HD$ and $G_\HA$, where $H$ denotes a ``full hypergraph''.
They give complete control over the feasibility and costs of composition changes, but they are also larger. Compare, e.g., the small \HypergraphModel{} snippets for Situation 1$'$ in \autoref{fig:ip difference hA} with its full hypergraph counterpart in \autoref{fig:ip difference HA}, in which the above-discussed transition of the red train unit is modeled by three separate arcs. One can, of course, mix these variants and resort to a composition expanded hypergraph only where necessary, as has been done, e.g., by \cite{Reuther2017}, but we will not do that in this paper for the sake of simplicity of exposition. 

\autoref{fig:strictIeq} also illustrates a third and final problem with fractional hyperflows in the small \HypergraphModel{}s, i.e., with the LP relaxation associated with these models. The snippets in \autoref{fig:strictIeqB} and \autoref{fig:strictIeqA} contain connections in which the mixed double train unit compositions must continue, while the pure red double train unit compositions are turned into mixed ones by uncoupling a train unit (in the example the second one) and replacing it by one of the other type that is pulled-out of the depot and coupled either to the front or to the back.
If that happens, coupling costs have to be paid that are assigned in the full \HypergraphModel{} to the respective composition change arcs, see \autoref{fig:strictIeqA}. These costs can be avoided in the small \HypergraphModel{} by a suitably chosen fractional hyperflow, see the bold (hyper)arcs in \autoref{fig:strictIeqB}. 
Namely, 0.5 times the two pure compositions can be connected to 0.5 times the two mixed compositions without paying coupling costs on pull-out/in arcs. 

For all four hypergraph variants $G_{\rm M}$, $\rm M\in\{\HD, \HA, \hD, \hA\}$, the rolling stock scheduling problem can be formulated as the following minimum cost hyperflow problem: 
\[\begin{array}{>{\rm}lr@{\;}c@{\;}ll}
  (\IPM)&\min c^T x \\
  (i)   &x(\delta^-(v)) - x(\delta^+(v)) &=& b_v & \forall\; v\in V \\
  (ii)  &x(H_t) &=& 1 & \forall\; t\in T \\
  (iii) &x(H_c) &=& 1 & \forall\; c\in C \\
  (iv)  &x&\geq& 0 \\
  (v)   &x_h&\leq& 1 & \forall\; h\in H\setminus H_D \\
  (vi) &x&\multicolumn{2}{l}{\rm integer}\\
\end{array}\]
Here, $x$ is the train unit hyperflow; it is binary, except possibly on the parking arcs in the depot timelines, which we denote by $H_D$. ($\IPM$) (i) are the flow conservation constraints; the node balances are all zero except for the beginning and the end of the depot timelines, where the balances are used to model train unit inventories. ($\IPM$) (ii) are the flow constraints; they ensure that every timetabled trip is serviced by a train of a single configuration. ($\IPM$) (iii) is special for our NS setting; it ensures a unique composition change for each connection and, in particular, that in each connection at least one train unit continues with a follow-on trip.
We will discuss their relevance later. ($\IPM$) (iv) are the non-negativity constraints, (v) the upper bounds; note that there is no upper bound on the parking arcs $H_D$ in the depot timelines. Finally, (vii) are the integrality constraints, and the objective minimizes the cost over the hyperarcs. 

\subsection{Modelling Rolling Stock Rotations by Composition Changes}

The \CompositionModel{} can be derived from the full \HypergraphModel{} $G_\HD$ by (i) contracting all nodes that differ only w.r.t. the composition, which contracts genuine hyperarcs into sets of parallel arcs that are then no longer node disjoint, (ii) deleting all but one of such parallel arcs, and (iii) deleting all arcs incident to the depot timelines, leaving the depot nodes isolated. 
Note that the initial node contraction in step (i) can produce parallel pull-in and pull-out arcs, which are all deleted in step (iii).
\autoref{fig:composition model graph} illustrates this construction to obtain the \CompositionModel{} for a regular 1-to-1 connection, where the depot nodes are just placeholders for their timestamps.

The above construction contracts sets of arrival/departure nodes associated with a composition $p$ into a single node, trip hyperarcs and turn hyperarcs for 1-to-1 connections into standard arcs that connect the associated compositions, and, finally, turn hyperarcs of 1-to-2 or 2-to-1 connections into hyperarcs that consists of two arcs that meet in a common node, i.e., these arcs are no longer node disjoint.
We denote the resulting hypergraph by $G_\C$. In this model, most hyperarcs will actually be arcs, only the turn hyperarcs for 1-to-2 and 2-to-1 connections are genuine (and, as mentioned above, non-graph based) hyperarcs. Each turn can involve a composition change, which in turn can involve pull-outs or pull-ins of train units.
These pull-outs must of course be possible, i.e., the respective train units must be available. 
To ensure that, one associates with every depot node $v$, for train units of type $r=r(v)$, the set $V_D^{\leq v}$ of all depot nodes that precede $v$ in time, up to the initial inventory node that we denote by $0_v$.
Moreover, one associates to $v$ sets of composition change (hyper)arcs $H^+_v$ and $H^-_v$ that imply a pull-out and pull-in of $\nu_h^r\in\N_0$ train units of type $r$ from or to this depot, respectively. 
With this notation, the number of available train units at depot node $v$ is $b_{0_v}-\sum_{h\in H^+_v} \nu^r_h x_h +\sum_{h\in H^-_v} \nu^r_h x_h$, which must not be negative.

\begin{figure}
	\centering
	\includegraphics[scale=1.4]{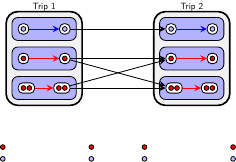}
	
	\caption{Illustration of the \CompositionModel{} (hyper)graph.}
	\label{fig:composition model graph}
\end{figure}

The rolling stock scheduling problem can be formulated in terms of the \CompositionModel{} as the following integer program:
\[\begin{array}{>{\rm}lr@{\;}c@{\;}ll}
  (\IPC)  &\min c^T x \\
  (i)     &x(\delta^-(v_p)) - x(\delta^+(v_p)) &=& 0 & \forall\; p\in P: \delta^-(p),\delta^+(p)\neq\emptyset\\
  (ii)    &x(H_t) &=& 1 & \forall\; t\in T \\
  (iii)   &x(H_c) &=& 1 & \forall\; c\in C \\
  (iv)    &b_{0_v}-\Sum_{h\in H^+_v} \nu^r_h x_h +\Sum_{h\in H^-_v} \nu^r_h x_h &\geq& 0 & \forall v\in V_D, r=r(v) \\
  (v)     &x&\geq& 0 \\
  (vi)    &x&\multicolumn{2}{l}{\rm binary}\\
\end{array}\]

Here, $x$ is the composition hyperflow, which is binary. ($\IPC$) (i) are the flow conservation constraints for all nodes $v_p$ that arise from a contraction of a composition $p$; some of these composition nodes have no predecessors or successors and hence no flow conservation is enforced for them. ($\IPC$) (ii) and (iii) ensure the selection of a unique composition for every trip and a unique composition change for follow-on trips in  connections; note that $H_t$ consists of arcs, while $H_c$ can contain hyperarcs for 1-to-2 and 2-to-1 connections. ($\IPC$) (iv) is the cut constraint described above that makes sure that the number of train units of type $r(v)$ doesn't exhaust the inventory at depot node $v$; note that one can enforce a certain end inventory by an artificial parking trip at the end of a timeline that consumes the required number of train units, such that no special constraint needs to be introduced. ($\IPC$) (v) and (vi) are the non-negativity and the integrality constraints, respectively.   

\subsection{The NS Setting: Model Harmonization}

In order to compare all models on a fair basis, we extend the NS setting by the following further assumptions:
\begin{itemize}
\item[v)] The cost of a direct connection arc $a$ in the connection arc models $G_\hA$ and $G_\HA$ is equal to the sum of the costs of the path $p_a$ that it represents in the depot models $G_\hD$ and $G_\HD$, i.e., $c_a=c(p_a)$ for all such arcs.
\item[vi)] There are no parking costs, and the costs of pull-in and pull-out arcs can be associated in a unique way with composition change hyperarcs, i.e., all costs can be associated with trip and composition change hyperarcs.
\end{itemize}
This is reasonable because the depot connection models $G_\hD$ and $G_\HD$, and even more the \CompositionModel{} $G_\C$, can not deal with costs that make finer distinctions. Let all other objective values be defined in the canonical way. When considering small \HypergraphModel{}s, which can not handle general costs for composition changes, we make an additional assumption:
\begin{itemize}
\item[vii)] The costs of hyperarcs in the composition expanded \HypergraphModel{}s $G_\HD$ and $G_\HA$ do not depend on the composition, i.e., the costs of two hyperarcs whose nodes differ only by compositions are the same.
\end{itemize}
Assumption vii) is made for the sake of the theoretical analysis. It is not always satisfied in our test set and its violation results in an underestimation of composition change costs in the small \HypergraphModel{}s.  

The logic of the \CompositionModel{} with its joins and splits via a depot is such that a train unit, once it has pulled-in to the depot, is available for arbitrary future pull-outs, i.e., the composition model does not have control over parked train units, and exactly the same holds for the depot connection variants of the \HypergraphModel{}. This is different in the arc connection variants of the \HypergraphModel{}, which can allow or forbid connections via the depot by introducing or not introducing the associated arc connections. These models are therefore stronger in this respect. To level the playing field, we consider the set $\olA$ as the closure of all possible arc connections, i.e., all arcs $a$ that shortcut a path $p_a$, and denote the associated arc connection models by $G_\holA$ and $G_\HolA$.
We will call such a setting unrestricted. 

\subsection{Model Comparison}

Denote by $\PIP(\rm M)$ and $\PLP(\rm M)$ the integer polyhedron and the LP relaxation, respectively, associated with model $\rm M\in\{\hD, \hA, \HD,\allowbreak \HA, \C\}$, and by $\valIP(\rm M)$ and $\valLP(\rm M)$ their optimal objective values.

\begin{theorem}\label{prop:inclusions} In an NS setting it holds for $\MP\in\{\LP,\IP\}$ that
\[\begin{array}[t]{c@{\;}c@{\;}c@{\;}c@{\;}c}
  \valMP(\HA) & \stackrel{a)}{\geq} & \valMP(\HD) & \stackrel{e)}{=} & \valMP(\C) \\[\medskipamount]
  {\scriptstyle c)}\; \rotatebox[origin=c]{90}{$\leq$} && \rotatebox[origin=c]{90}{$\leq$}\; {\scriptstyle d)}\\
  \valMP(\hA) & \stackrel{b)}{\geq} & \valMP(\hD)\rlap{.} \\
\end{array}\]
Here,  a) and b) hold with equality if $A=\olA$, i.e., if the setting is unrestricted. Moreover, a) and e) hold independently of assumption vii). 
\end{theorem}
\begin{proof}~
\begin{itemize}
  \setlength{\itemsep}{1pt}
  \setlength{\parskip}{0pt}
  \setlength{\parsep}{0pt}
  \item[a)] $A\subseteq\olA$ implies $\valMP(\HA)\geq\valMP(\HolA)$, hence it suffices to show that a) holds with equality for this case.
  Given that, let $\MP=\LP$ and $y\in\PLP(\HA)$. For each (hyper)arc $h\in H(G_\HD)$, define 
  \[
    x_h=\begin{cases} y_h, & h\in H\setminus H_D \\ \sum_{p_{a} \ni h} y_{a}, & h\in H_D, \end{cases} 
  \]
  i.e., route the flow $y_{a}$ on a direct connection arc $a$ along the pull-in, parking, and pull-out arcs of the associated path $p_{a}$. Then $x$ belongs to $\PLP(\HD)$ and, by assumption vii), produces the same cost as $y$. As this argument works also for integer $y$, which produces an integer $x$, the case $\MP=\IP$ holds as well.
  \item[b)] Analogous to a).
  \item[c)] Let $y\in\PLP(\HA)$ and denote for a (hyper)arc $h\in H(G_\hA)$ by $H^P_h$ the set of (hyper)arcs in $H(G_\HA)$ that project w.r.t.\ composition onto $h$, i.e., these hyperarcs are copies of $h$ that  (possibly) connect different compositions. For $h\in H(G_\hA)$, define 
  \[
    x_h=\sum_{g\in H^P_h} y_{g}, 
  \]
  i.e., aggregate the flow values on all copies. Then $x$ belongs to $\PLP(\hA)$ and, by assumption v), produces the same cost as $y$. As this argument works also for integer $y$, which produces an integer $x$, the case $\MP=\IP$ holds as well. Finally, note that assumption vii) was not used. 
  \item[d)] Analogous to c). 
  \item[e)] We note that the hyperarcs of the two models are in 1-1 correspondence. Any hyperflow $y$ in $\PLP(\HD)$ gives rise to a feasible composition flow $x=y_{H\setminus H_D}$ by simply omitting the depot arcs, and any composition flow $x\in\PLP(\C)$ can be extended to a hyperflow in $G_\HD$ by supplementing the appropriate pull-outs, parkings, and pull-ins, which can be done because of the max-flow min-cut Theorem. By assumption vi), these flows have the same costs. As this argument works also for integer $y$, which produces an integer $x$, the case $\MP=\IP$ holds as well. Assumption vii) was again not used. 
\end{itemize}
\end{proof}

We already gave examples that the inequalities c) and d) can be strict, i.e., the small \HypergraphModel{}s are in general weaker than the full, composition extended ones.
In fact, we showed that they can allow for infeasible solutions and underestimate coupling costs, i.e., they are genuine (over-)relaxations. 
In reality, at least in our test set, the difference is not as bad as one might expect, but not because these effects are rare, just because their overall cost impact is small. 
In fact, the situations depicted in the snippets in \autoref{fig:ip difference hA} and \autoref{fig:strictIeqA} are frequent in our test set.

We note the following consequence of Theorem~\ref{prop:inclusions}.

\begin{corollary}\label{cor:projection} In an NS setting it holds for $\MP\in\{\LP,\IP\}$ that
\[
  P_\MP(\HolA)|_{H\setminus\olA}=P_\MP(\HD)|_{H\setminus H_D} = P_\MP(\C).
\]
\end{corollary}

In other words, the \CompositionModel{} is the projection of the unrestricted full \HypergraphModel{} onto the subspace of non-depot hyperarcs, both in an LP and an IP sense. This is remarkable, as projections of polyhedra involve a Fourier-Motzkin elimination, that usually does not result in an explicitly known description, not to speak of a description that is combinatorially meaningful. In this case, the projection eliminates the train unit flow for depot connections, and replaces primal flow conservation by a dual formulation in terms of cuts. 

\subsection{The Connection Constraints}

Theorem~\ref{prop:inclusions} hinges on the presence of the flow constraints $\IPM$ (iii) that stipulate the choice of exactly one composition change on every connection in all models. This constraint is special for the NS setting and the \CompositionModel{} but uncommon in the literature on the \HypergraphModel{}. If this constraint is omitted from the \HypergraphModel{}s (it cannot be omitted from the \CompositionModel{}), they lose theoretical strength, as the example in \autoref{fig:omitting_the_flow_constraints} shows. In the \CompositionModel{}, the mixed composition can only continue or split one train unit off, but the mixed composition cannot be turned around. This is, however, possible in the \HypergraphModel{} (in the version without connection constraints) by sending both train units to the depot, and pulling out two new train units to assemble the reverse composition without using a composition change hyperarc. 
While initially surprising, these connection constraints actually turn out to be redundant in our real-world test set, i.e., the instances are somewhat  restrictive and simply do not contain reasonable options to not use exactly one composition change.

\begin{figure}[htbp]
  \centering
  \begin{subfigure}[b]{0.49\textwidth}
    \centering
    \includegraphics[page=1,scale=1.05]{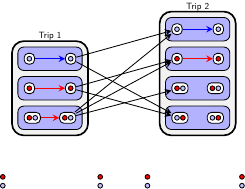}
    \caption{Composition Model}
  \end{subfigure}
  \begin{subfigure}[b]{0.49\textwidth}
    \centering
    \includegraphics[page=2,scale=1.3]{./tikzpictures/LPGapCompositionHypergraph.pdf}
    \caption{Full Hypergraph Model with Depot Connections}
  \end{subfigure}
  \caption{Omitting the flow constraints on composition changes weakens the \HypergraphModel{}s.}
  \label{fig:omitting_the_flow_constraints}
\end{figure}

\subsection{The NS Setting: Complexity}

The NS setting is sufficiently restrictive to suggest that it might have a different computational complexity than general multicommodity flow problems. This question has been considered by \cite{alfieri2006efficient}, who show the NP-hardness of a version of the problem in which feasible compositions are defined implicitly by numerical passenger demands. They provide a reduction from numerical three-dimensional matching that relies on an overall exact match of train capacities and demands, which is a rare scenario, and, as this argument involves numbers, does not show strong NP-hardness. The proof also uses a very large number of train unit types ($2n+1$ types for $3n$ trains and $6n$ trips), whereas in reality, the number of train unit types is a small constant. In fact, \cite{alfieri2006efficient} consider a problem with two train unit types, and the same holds for all the real world instances of our computational study. \cite{alfieri2006efficient} also consider a version with explicit compositions that is identical with our setting, but do not study its complexity. As the number of explicit compositions is in general not polynomial in numerical demand data, the complexity status of this problem is not necessarily the same, and has therefore, to the best of our knowledge, not been determined up to now. On the other hand, \cite{alfieri2006efficient} show that the problem can be solved in polynomial time if all parameters except the number of trips are constant. We will show now that the problem with explicit compositions is NP-hard in the NS-setting, in fact, that it is NP-hard in the strong sense, even for a train unit fleet consisting out of a single train unit type.

\begin{figure}[ht]
  \centering
  \includegraphics[width=0.99\textwidth]{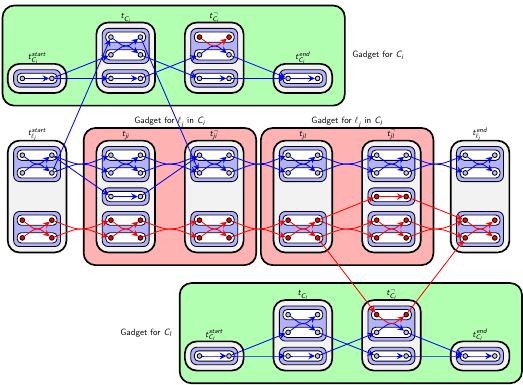}
  \caption{Reducing 3SAT to NS-\ProblemAcronym{}. The red and blue colors only highlight true and false rotations, they do not encode train unit types (there is only one type).}
  \label{fig:complexity2}
\end{figure}

\begin{theorem}\label{prop:complexity1}
  The \ProblemName{} in the NS setting (NS-\ProblemAcronym{}) is strongly NP-hard, even for a single train unit type. 
\end{theorem}
\begin{proof}
    \newcommand\rmstart{\text{start}}
    \newcommand\rmend{\text{end}}
    \newcommand\rmtrue{\text{}}
    \newcommand\rmfalse{\text{$\neg$}}
    The proof is by reduction from 3SAT. Consider an instance  $\bigwedge_{i\in I} C_i, C_i=\bigvee_{j\in J_i} \sign_{ij}\ell_j$ of 3SAT with clauses $C_i, i\in I$, and literals $\ell_j, j\in J$, that can be negated or unnegated according to their sign $\sign_{ij}$, each clause containing three of them, i.e., $|J_i|\equiv 3$.
    Let $I=[m]$ and $J=[n]$.
    
    In the corresponding NS-\ProblemAcronym{}, all train units are of the same type. There is a ``clause train'' that operates for each clause $C_i$ a gadget of four trips $t_{C_i}^\rmstart$, $t_{C_i}^\rmtrue$, $t_{C_i}^\rmfalse$, and $t_{C_i}^\rmend$ that commute between two stations $A$ and $B$, see \autoref{fig:complexity2} for an illustration.
    These trips are follow-on trips of each other in the given order and, moreover, $t_{C_{i+1}}^\rmstart$ is the follow-on trip of $t_{C_i}^\rmend$, $1\leq i<m$. Trip $t_{C_i}^\rmstart$ must be operated by a unique single train unit composition by a ``clause train unit''. The follow-on trip $t_{C_i}^\rmtrue$ and its follow-on trip $t_{C_i}^\rmfalse$ deal with the satisfaction of clause $C_i$, which we will attribute to one of its three literals (if two literals satisfy a clause, one of them can be chosen arbitrarily). This satisfying literal must be set to true if it is unnegated and to false if it is negated. Consider first the unnegated literals of clause $C_i$. One of these literals can satisfy the clause or none of them does. For the first case, provided $C_i$ contains an unnegated literal, there is a double train unit composition that pulls-in a second train unit (which will come from a ``true rotation'' of associated literal train units) in addition to the clause train unit. For the second case, there is a single train unit composition on which the clause train unit continues; if $C_i$ does not contain unnegated literals, this single train unit composition is the only one for trip $t_{C_i}$. These (one or two) compositions are successors of the single train unit composition on $t_{C_i}^\rmstart$. The same compositions are available for the trip $t_{C_i}^\rmfalse$, only this time the double train unit composition represents the satisfaction of clause $C_i$ by a negated literal (train unit). Now the construction is such that the double train unit composition on $t_{C_i}^\rmtrue$ must be followed by the single train unit composition on $t_{C_i}^\rmfalse$, and the single train unit composition on $t_{C_i}^\rmtrue$ by the double train unit composition on $t_{C_i}^\rmfalse$ s.t.\ exactly one of the two double train unit compositions is chosen. Finally, the (two or one, depending on whether $C_i$ contains a negated literal or not) compositions for $t_{C_i}^\rmfalse$ are both linked to a unique single train unit composition on $t_{C_i}^\rmend$, forcing the clause train unit to proceed through all four clause trips in sequence, and on to the next clause trip $t_{i+1}^\rmstart$ (if it exists). 
    
    There is also a ``literal train'' for every literal $\ell_j$. Its trips also commute between stations $A$ and $B$. There is an initial trip $t_{\ell_j}^\rmstart$, 
    followed by a sequence of follow-on trips through several gadgets, namely, one gadget for every clause $C_i$ in which $\ell_j$ appears.
    Each of these gadgets operates two trips $t_{ji}$ and $t_{ji}^\neg$ that operate at the same times as the clause trips $t_{C_i}$ and $t_{C_i}^\neg$ to which they correspond (we write $t_{ji}$ instead of $t_{\ell_jC_i}$ for lighter notation).
    These clause gadget trips are followed by a final trip $t_{\ell_j}^\rmend$.
    Each trip of such a literal train must be operated by one out of two double train unit compositions, which represent the setting of literal $\ell_j$ to true or false. With one exception in every set of clause gadget trips, the true composition of a literal trip is only linked to the true composition of the follow-on trip, and the false composition to the false composition of the follow-on trip, such that the two train units of the literal train must follow either an all true or an all false rotation. The one exception in every set of clause gadget trips depends on whether literal $\ell_j$ appears in clause $C_i$ in unnegated or negated form. Suppose the first; then $\ell_j$ can satisfy clause $C_i$ if it is set to true. To model this option, there is an additional single train unit composition for trip $t_{ji}$, that allows one of the literal train units to pull-out for the time of trip $t_{ji}$, which is also the time of the clause trip $t_{C_i}$, and to pull-in to the double train unit composition that satisfies this clause. After the trip, the train unit must pull-out of the clause composition again, and return to the literal double train unit composition for trip $t_{ji}^\neg$ to continue. If literal $\ell_j$ appears in clause $C_i$ in negated form, it can satisfy the clause if it is set to false. In this case, trip $t_{ji}^\neg$ offers a single train unit composition that allows a literal train unit to pull-out into the double train unit composition for trip $t_{C_i}^\neg$ that satisfies clause $C_i$. The following double train unit composition for the follow-on trip in the next clause gadget, or in the final trip $t_{\ell_j}^\rmend$, forces the train unit to return to the literal rotation of its truth value. 
    
    In this way, a satisfying truth assignment gives rise to a feasible rolling stock rotation schedule, and the converse also holds. The construction is polynomial and does not involve any numbers.
\end{proof}

The proof is illustrated in \autoref{fig:complexity2}. Stations $A$ and $B$ appear alternatingly in this figure. At the top and bottom are the four trips of two clauses $C_i$ and $C_l$, respectively. In the middle are the trips of the associated literal gadgets.  Here, the literal $\ell_j$ is unnegated in clause $C_i$, and negated in clause $C_l$. This can be seen by looking at the associated pull-in and pull-out movements, where pull-outs and corresponding pull-ins occur to trip $t_{C_i}$ and trip $t_{C_l}^\neg$.

\section{Numeric Model Comparison}
\label{sec: numerical comparison}
In this section, we complement the analytic comparison of the \CompositionModel{} and the different variants of the \HypergraphModel{} by a numerical comparison on a set of real-world instances from NS.
We are particularly interested in evaluating the performance of the linear programming bound provided by these models and their relative performance in terms of the time required   to find (near-)optimal solutions.
In the remainder of this section, we first introduce the test set and then discuss the numerical results.

\subsection{Instances}
Our test set is based on the 2018 timetable operated by NS, which is the largest passenger railway operator in the Netherlands.
\Cref{fig: dutch railway network} illustrates the network.
NS operates both intercity services and regional (sprinter) services, where the latter generally stop at each encountered station along a railway line.
Both intercity and sprinter trains operate in a relatively high frequency, with usually 2 to 4 trains per hour for each timetable service.

\begin{table}[htbp]
  \centering
  \begin{minipage}[b]{0.49\textwidth}
    \centering
    \includegraphics[scale=0.3]{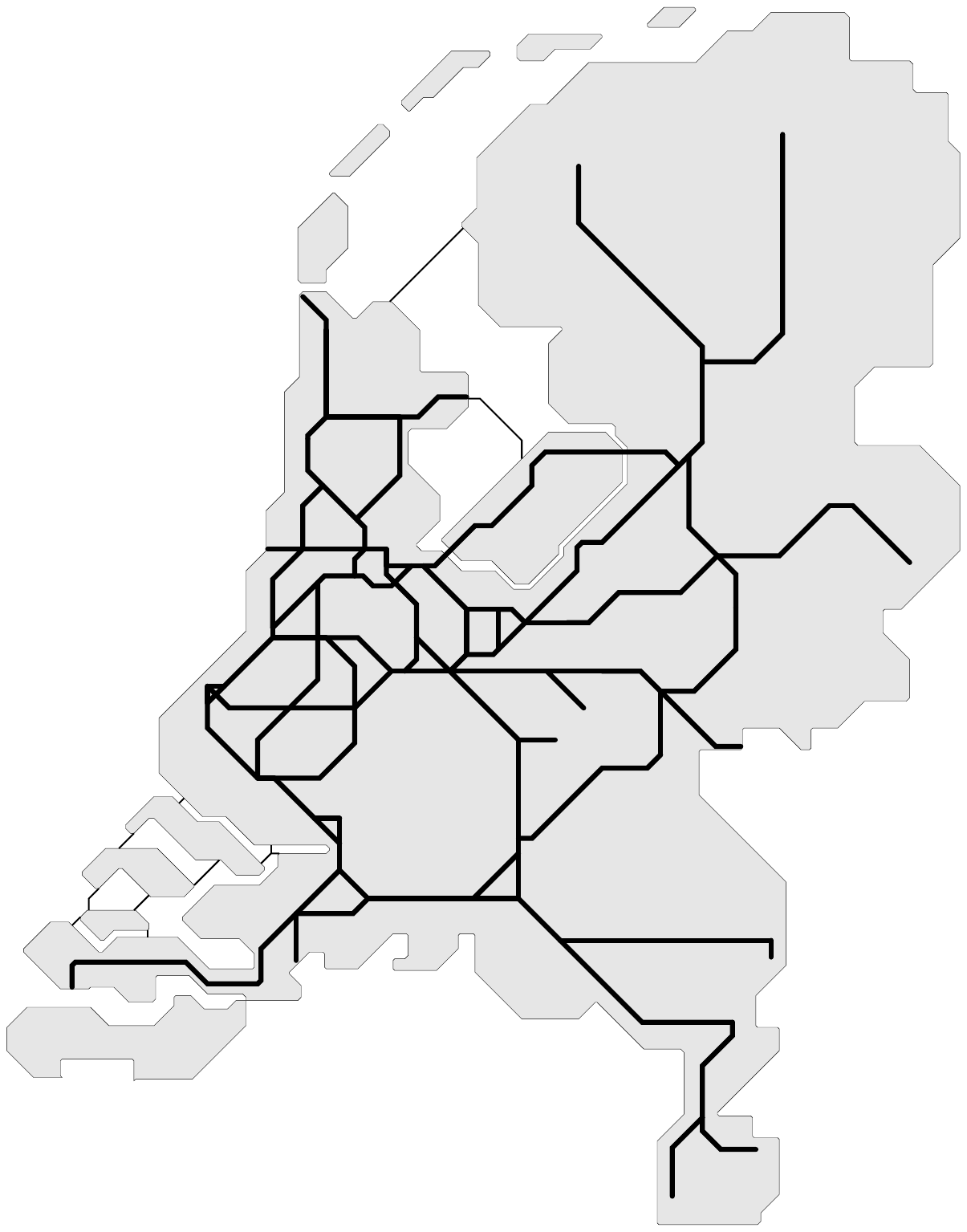}
    \captionof{figure}{Overview of the Dutch railway network}
    \label{fig: dutch railway network}
  \end{minipage}
  \begin{minipage}[b]{0.49\textwidth}
  \nextfloat
  
  \begin{subtable}[c]{0.99\textwidth}	
    \centering    
    \begin{tabular}{lccc}
			\toprule
			Instance & $|\TripSet|$ & $|\TypeSet|$ & $|P|$ \\
			\midrule
   			FLIRT(-AN) & 584 (+1) & 2 & 14 \\
   			SGM(-AN) & 758 & 2 & 14 \\
			SLT & 1681 & 2 & 14 \\
   			DDZ(-AN) & 419 & 2 & 6 \\
      		VIRM(-AN) & 1300 & 2 & 10 \\
			ICM & 1222 & 2 & 30 \\
			\bottomrule
	\end{tabular}
	\caption{Statistics of the considered instances.}
	\label{tab: statistics instances}
  \end{subtable}
  \vskip1em
  \begin{subtable}[c]{0.99\textwidth}
    \centering
    \begin{tabular}{lr}
    	\toprule
    	Parameter & Value \\
    	\midrule
    	Mileage & 0.1 \\
     	Seat shortage & 0.2 \\
     	Shunting & 10 \\
     	Ending deviation & 10000 \\
    	\bottomrule
    \end{tabular}
    \caption{Parameters of the objective function.}
    \label{tab: objective values}
  \end{subtable}
  
  \caption{Test set statistics.}
  \label{tab: test set}
  \end{minipage}

\end{table}

We create rolling stock scheduling instances by considering different rolling stock fleets.
At NS, train unit types are categorized into so-called train unit families.
Only train unit types that belong to the same family can be combined in a composition.
We consider train units of six families: FLIRT, SGM, SLT, DDZ, VIRM, and ICM.
The first three families are used to operate sprinter trains, while the latter three are (mainly) used for intercity services.
All train unit families are composed of two different rolling stock types, i.e., contain train units of two types.
Instances are then created by only considering those trips in the timetable that have, in an earlier step, been designated to be operated by train units of the considered family.
For the train unit families FLIRT, SGM, DDZ, and VIRM, we have two slightly different input specifications available, especially deviating in the number of train units that should end at each station, leading to a total of 10 instances.
Each instance is named according to the train unit family it represents, and the two different instances for the FLIRT, SGM, DDZ, and VIRM families are distinguished by adding ``-AN'' as a subscript in the name for each second instance.

Summary statistics about the instances are shown in \autoref{tab: statistics instances}, where the number of trips in the instance, the number of considered rolling stock types, and the number of possible compositions are given.
Moreover, the underlying trips in the instances are visualized in \autoref{fig:FleetTimetables}.
Based on the summary statistics, DDZ(-AN), FLIRT(-AN) and SGM(-AN) are the smallest instances as they consider the fewest number of trips.
In addition, it can be seen from the visualizations that the trips operated by SGM train units geographically decompose into two parts and that the DDZ instance largely focuses on one corridor in the country.
The SLT instance is the largest in the number of trips and can be seen from the visualization to cover a large part of the Netherlands.
On the other hand, ICM train units can be coupled in longer compositions of up to five train units, giving a larger number of possible compositions than in the other instances.
The VIRM(-AN) instance is in between the SLT and ICM instance in terms of the number of trips and can be seen to especially cover services in the south and western parts of the country.

\begin{figure}[htbp!] 
	\centering
	
	\begin{subfigure}[t]{0.49\textwidth}
		\centering
		\includegraphics[height=0.28\textheight]{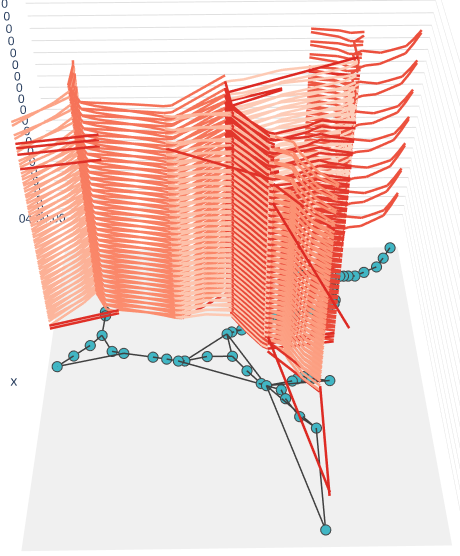}
		\caption{FLIRT}
		\label{fig:FLIRTtimetable}
	\end{subfigure}
	\begin{subfigure}[t]{0.49\textwidth}
		\centering
		\includegraphics[height=0.28\textheight]{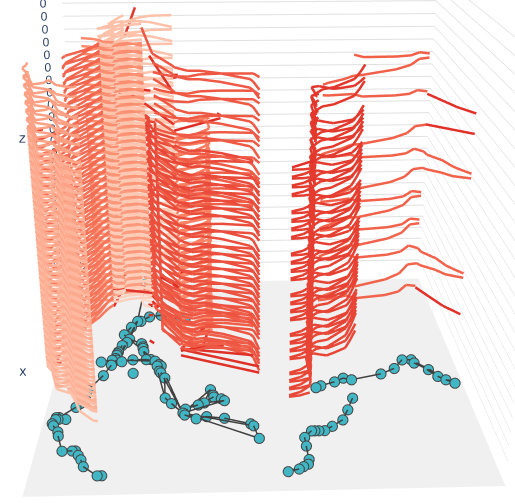}
		\caption{SGM}
		\label{fig:SGMtimetable}
	\end{subfigure}
	
	\begin{subfigure}[t]{0.49\textwidth}
		\centering
		\includegraphics[height=0.28\textheight]{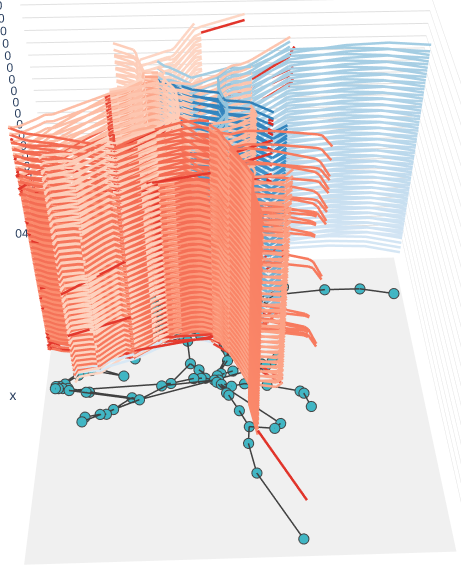}
		\caption{SLT}
		\label{fig:SLTtimetable}
	\end{subfigure}
	\begin{subfigure}[t]{0.49\textwidth}
		\centering
		\includegraphics[height=0.28\textheight]{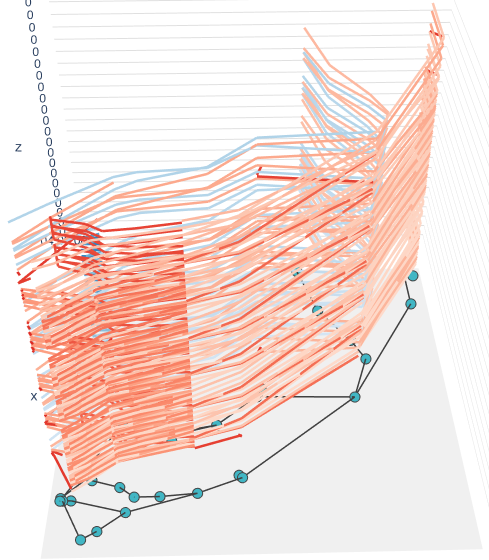}
		\caption{DDZ}
		\label{fig:DDZtimetable}
	\end{subfigure}
	
	\begin{subfigure}[t]{0.49\textwidth}
		\centering
		\includegraphics[height=0.28\textheight]{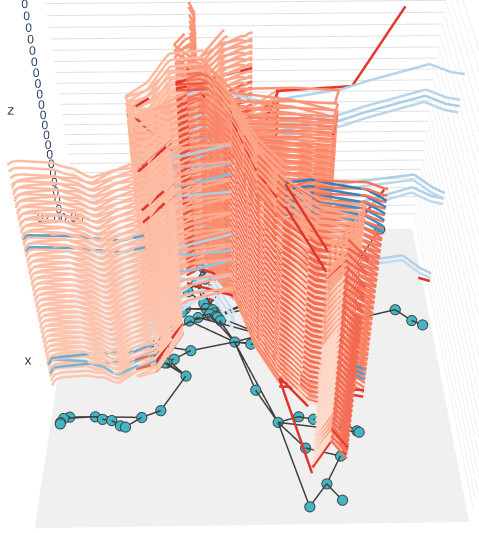}
		\caption{VIRM}
		\label{fig:VIRMtimetable}
	\end{subfigure}
	\begin{subfigure}[t]{0.49\textwidth}
		\centering
		\includegraphics[height=0.28\textheight]{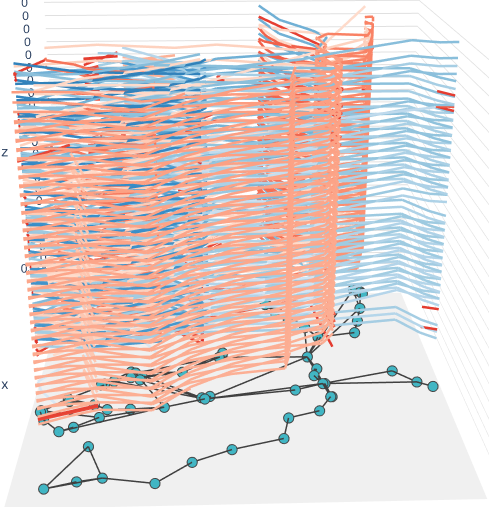}
		\caption{ICM}
		\label{fig:ICMtimetable}
	\end{subfigure}
	
	\caption{Visualization of the timetables underlying the NS instances. Each plot shows the trips in the timetable in space (horizontal axes) and time (vertical axis). Moreover, the color of a line indicates the cheapest composition from the set of suitable compositions to operate the trip in the timetable and the line's thickness indicates this composition's size.}
	\label{fig:FleetTimetables}
\end{figure}

The parameters used for the objective function can be found in \autoref{tab: objective values}.
The first two parameters relate to the operated compositions.
These define the costs per traveled kilometer of a carriage and per seat that is short compared to the expected number of passengers, respectively.
Note that these costs can be directly associated with composition hyperarcs in the \HypergraphModel{} variants and with composition variables in the \CompositionModel{}.
The third parameter relates to the cost of performing a shunting action, which is incurred if either coupling or uncoupling occurs during a connection.
These costs can be directly associated with composition change variables in the \CompositionModel{}, as well as with composition change hyperarcs in the full \HypergraphModel{} variants due to both the predecessor composition and successor composition of such a hyperarc being known in these variants.
The latter is not possible in the small \HypergraphModel{} variants, meaning that these models generally underestimate the true shunting cost. 
The fourth parameter defines the cost of a deviation from the number of train units, of a given rolling stock type, that are expected to end at a station.
These costs can be easily taken into account in all models by including extra variables that measure the deviation to the ending node balances. 

\subsection{Results}
For the numerical comparison, we implemented all models as pure MILP models and solved them with the CPLEX 20.1.0 general-purpose MILP solver.
We have chosen this approach, over using the exact solution methodology used to solve the \CompositionModel{} and \HypergraphModel{} at the two companies, as both are highly specialized for the individual company's needs.
Moreover, the column generation methodology developed for the  \HypergraphModel{} does not easily translate to all the variants of this model.

\paragraph{Model Sizes}
The size of the resulting MILP models, in terms of the number of variables and constraints, is given in \autoref{tab: model sizes} and illustrated in \autoref{fig:sizes} for each of the NS instances.

It can be seen that the small \HypergraphModel{}s \hA{} and \hD{} have the fewest number of constraints, where the direct connection model \hA{} is slightly smaller than the depot model \hD{}. 
The smaller number of constraints for these models can be explained by the contraction of nodes leading to fewer flow conservation constraints, where the direct connections in \hA{} additionally prevent the constraints needed to model the flow in the depot.
The \CompositionModel{} \C{} is close in the number of constraints to the small Hypergraph models due to modeling flow conservation on the composition instead of train unit level.
The full Hypergraph models \HA{} and \HD{} are clearly the largest in terms of the number of constraints, which can be explained by the duplication of nodes for the positions in the compositions and the corresponding duplication of flow conservation constraints.
This especially leads to a large number of constraints for the ICM instance, which allows for a larger number of possible compositions per trip.

Compared to the number of constraints, the differences between the models are more pronounced for the number of variables.
Note that the $y$-axis is given on a logarithmic scale in \autoref{fig:nvars}. 
First, it can be seen that the models that track connections directly (\hA{} and \HA{}) are significantly larger than those models that use a depot representation (\hD{} and \HD{}) or track the inventories implicitly (\C{}).
This effect is especially clear for the ICM instance, which allows for longer compositions and thus provides more direct connection possibilities.
Clearly, the more explicit the rolling stock connections are handled, the larger the model grows in terms of the number of variables.
Second, it can be seen that the full \HypergraphModel{}s \HD{} and \HA{} use about 1.1 -- 1.3 and 1.6 -- 2.8 as many variables as their direct counterparts \hD{} and \hA{}, respectively, and even more for instance ICM.
Hence, the duplication of nodes over compositions leads to a considerable increase in the number of hyperarcs, in particular when there are many feasible compositions as in the ICM instance.
Overall, it can be concluded that \hA{} is the smallest model on all instances in terms of the number of variables, closely followed by the \CompositionModel{}.
The full \HypergraphModel{} \HD{} lags not far behind, followed by a relatively large gap towards the \HypergraphModel{}s \hD{} and \HD{} that use direct connections.

\begin{table}[htbp]
	\caption{Size of the Composition and Hypergraph models, expressed in the number of variables (\textit{Var.}) and constraints (\textit{Cons.}), for each of the NS instances.}
	\label{tab: model sizes}
	
	\resizebox{\textwidth}{!}{
		\begin{tabular}{lcccccccccc}
                \toprule
                & \multicolumn{2}{c}{C} & \multicolumn{2}{c}{\HD} & \multicolumn{2}{c}{\hD} & \multicolumn{2}{c}{\HA} & \multicolumn{2}{c}{\hA} \\
                \cmidrule(rl){2-3} \cmidrule(rl){4-5} \cmidrule(rl){6-7} \cmidrule(rl){8-9} \cmidrule(rl){10-11}
                instance & Var. & Cons. & Var. & Cons. & Var. & Cons. & Var. & Cons. & Var. & Cons. \\
                \midrule
                FLIRT-AN &    7561 &    7941 &    8244 &   11757 &    7346 &    6803 &   15267 &   11151 &    9602 &    5671 \\
                FLIRT    &    7579 &    7944 &    8291 &   11783 &    7347 &    6795 &   15764 &   11182 &    9657 &    5668 \\
                SGM-AN   &   12899 &   12418 &   14875 &   20511 &   12360 &   10320 &   53238 &   19699 &   22891 &    8843 \\
                SGM      &   12899 &   12418 &   14875 &   20513 &   12360 &   10322 &   53296 &   19699 &   22887 &    8843 \\
                SLT      &   21531 &   22818 &   23499 &   33794 &   20837 &   19678 &   58477 &   31668 &   32161 &   16366 \\
                DDZ-AN   &    5905 &    5386 &    7135 &    7775 &    5549 &    4653 &   30048 &    7328 &   13174 &    3868 \\
                DDZ      &    7961 &    8075 &    9298 &   12061 &    7410 &    6703 &   36050 &   10769 &   16148 &    5073 \\
                VIRM-AN  &   18717 &   17590 &   22778 &   27121 &   18128 &   15606 &  109903 &   25627 &   39413 &   13153 \\
                VIRM     &   20477 &   18780 &   24975 &   29582 &   19665 &   16351 &  120768 &   27995 &   43949 &   13805 \\
                ICM      &   42449 &   29431 &   60551 &   67746 &   34014 &   17182 & 2320477 &   66231 &  158006 &   14855 \\
                \bottomrule
            \end{tabular}}
\end{table}

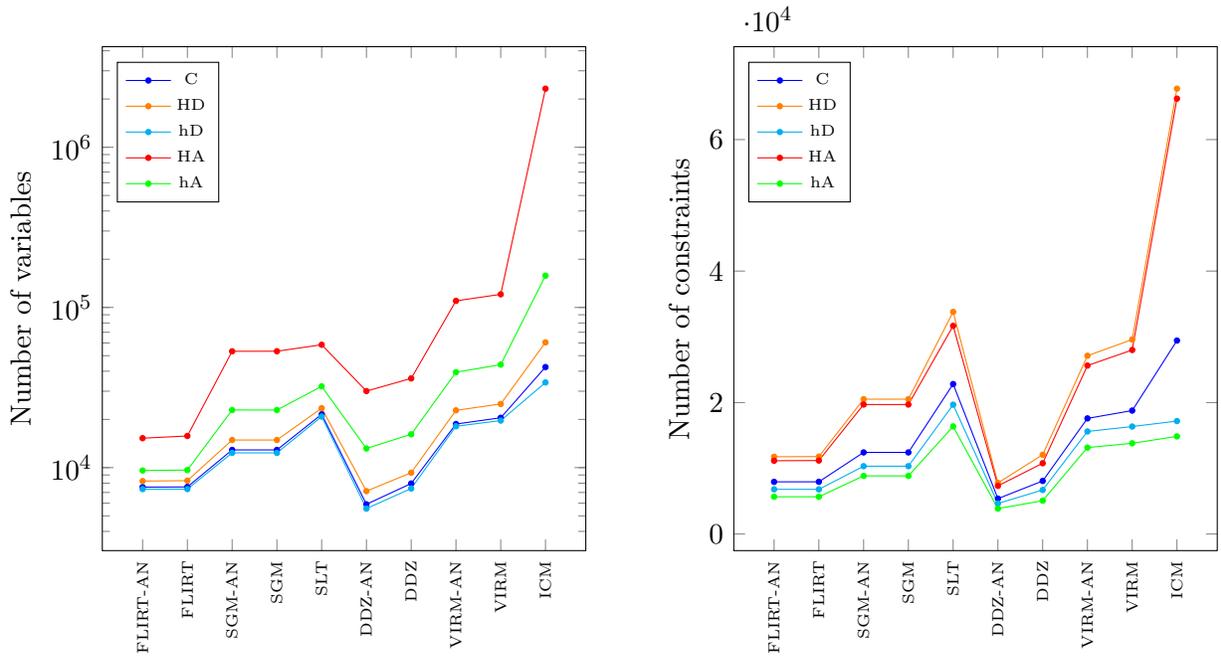
\begin{figure}[htbp]
  \centering%
  \begin{subfigure}[b]{0.49\textwidth}
    \begin{tikzpicture}
    \begin{semilogyaxis}[
      width=0.99\textwidth,
      height=0.38\textheight,
      ylabel={Number of variables},
      symbolic x coords={FLIRT-AN,FLIRT,SGM-AN,SGM,SLT,DDZ-AN,DDZ,VIRM-AN,VIRM,ICM},
      x tick label style={rotate=90, font=\tiny},
      xtick=data,
      legend style={legend pos=north west,legend columns=1, font=\tiny},
    ]
      \addplot[blue, mark=*, mark size=1pt] table[x=instance,y=C-V,col sep=comma]{tables/size_table.csv};
      \addplot[orange, mark=*, mark size=1pt] table[x=instance,y=HD-V,col sep=comma]{tables/size_table.csv};
      \addplot[cyan, mark=*, mark size=1pt] table[x=instance,y=hD-V,col sep=comma]{tables/size_table.csv};
      \addplot[red, mark=*, mark size=1pt] table[x=instance,y=HA-V,col sep=comma]{tables/size_table.csv};
      \addplot[green, mark=*, mark size=1pt] table[x=instance,y=hA-V,col sep=comma]{tables/size_table.csv};
      \legend{$\C$, $\HD$, $\hD$, $\HA$, $\hA$}
    \end{semilogyaxis}
    \end{tikzpicture}
    \caption{Model sizes in terms of numbers of variables.}
    \label{fig:nvars}
  \end{subfigure}
  \hfill
  \begin{subfigure}[b]{0.49\textwidth}
  	 \centering
    \begin{tikzpicture}
    \begin{axis}[
      width=0.99\textwidth,
      height=0.38\textheight,
      ylabel={Number of constraints},
      symbolic x coords={FLIRT-AN,FLIRT,SGM-AN,SGM,SLT,DDZ-AN,DDZ,VIRM-AN,VIRM,ICM},
      x tick label style={rotate=90, font=\tiny},
      xtick=data,
      legend style={legend pos=north west,legend columns=1,font=\tiny},
    ]
      \addplot[blue, mark=*, mark size=1pt] table[x=instance,y=C-C,col sep=comma]{tables/size_table.csv};
      \addplot[orange, mark=*, mark size=1pt] table[x=instance,y=HD-C,col sep=comma]{tables/size_table.csv};
      \addplot[cyan, mark=*, mark size=1pt] table[x=instance,y=hD-C,col sep=comma]{tables/size_table.csv};
      \addplot[red, mark=*, mark size=1pt] table[x=instance,y=HA-C,col sep=comma]{tables/size_table.csv};
      \addplot[green, mark=*, mark size=1pt] table[x=instance,y=hA-C,col sep=comma]{tables/size_table.csv};
      \legend{$\C$, $\HD$, $\hD$, $\HA$, $\hA$}
    \end{axis}
    \end{tikzpicture}
    \caption{Model sizes in terms of numbers of constraints.}
    \label{fig:ncons}
  \end{subfigure}
  \caption{Model sizes for all models and instances.}
  \label{fig:sizes}
\end{figure}

\paragraph{Obtained Bounds}
The objective value and linear programming relaxation bound obtained by each model are given in \autoref{tab: model bounds} and plotted in \autoref{fig:data} for each of the NS instances.
The colors in \autoref{tab: model bounds} group together models that obtain the same values for a given instance. 

\newcommand\opt{\color{blue}}
\newcommand\agg{\color{cyan}}
\newcommand\thr{\color{green}}
\begin{table}[htbp]
	\caption{Obtained linear programming relaxation (\textit{LP}) and integer objective (\textit{MILP}) value for the Composition and Hypergraph models for each of the NS instances.}
	\label{tab: model bounds}
	\resizebox{\textwidth}{!}{
\begin{tabular}{l*{10}r}
\toprule
& \multicolumn{2}{c}{C} & \multicolumn{2}{c}{\HD} & \multicolumn{2}{c}{\hD} & \multicolumn{2}{c}{\HA} & \multicolumn{2}{c}{\hA} \\
\cmidrule(rl){2-3} \cmidrule(rl){4-5} \cmidrule(rl){6-7} \cmidrule(rl){8-9} \cmidrule(rl){10-11}
instance & LP & MILP & LP & MILP & LP & MILP & LP & MILP & LP & MILP \\
\midrule
    FLIRT-AN & \opt18805.3 & \opt18992.2 & \opt18805.3 & \opt18992.2 & \agg18306.2 & \agg18306.2 & \opt18805.3 & \opt18992.2 & \agg18306.2 & \agg18306.2 \\
    FLIRT & \opt17074.4 & \opt17074.4 & \opt17074.4 & \opt17074.4 & \agg16820.8 & \agg16820.8 & \opt17074.4 & \opt17074.4 & \agg16820.8 & \agg16820.8 \\
    SGM-AN & \opt19630.3 & \opt19646.8 & \opt19630.3 & \opt19646.8 & \agg18924.4 & \agg18939.8 & \opt19630.3 & \opt19646.8 & \agg18924.4 & \agg18939.8 \\
    SGM & \opt18998.6 & \opt18998.6 & \opt18998.6 & \opt18998.6 & \agg18290.8 & \agg18299.7 & \opt18998.6 & \opt18998.6 & \agg18290.8 & \agg18299.7 \\
    SLT & \opt 35291.7 & \opt 35291.7 & \opt 35291.7 & \opt 35291.7 & \agg34339.3 & \agg34392.1 & \opt35291.7 & \opt35291.7 & \agg34339.3 & \agg34392.1 \\
    DDZ-AN & \opt20040.2 & \opt20040.2 & \opt20040.2 & \opt20040.2 & \agg 19526.4 & \agg 19580.2 & \opt20040.2 & \opt20040.2 & \agg19526.4 & \agg19580.2 \\
    DDZ & \opt18427.4 & \opt18427.4 & \opt18427.4 & \opt18427.4 & \agg17913.6 & \agg17957.4 & \opt18427.4 & \opt18427.4 & \agg17913.6 & \agg17957.4 \\
    VIRM-AN & \opt98297.4 & \opt98732.7 & \opt98297.4 & \opt98732.7 & \agg95476.8 & \agg96503.9 & \opt98297.4 & \opt98732.7 & \agg95476.8 & \agg96503.9 \\
    VIRM & \opt86941.0 & \opt86947.4 & \opt86941.0 & \opt86947.4 & \agg84028.7 & \agg84459.2 & \opt86941.0 & \opt86947.4 & \agg84028.7 & \agg84459.2 \\
    ICM & \opt 39851.4 & \opt 39885.7 & \opt 39851.4 & \opt 39885.7 & \agg36623.4 & \agg37011.4 & \opt39851.4 & \opt39885.7 & \thr36641.2 & \thr37032.2 \\
    \bottomrule
\end{tabular}}
\end{table}

\begin{figure}[htbp]
  \centering%
  \begin{subfigure}[b]{0.49\textwidth}
	  \centering
	  \begin{tikzpicture}
	  \begin{axis} [%
	    width=0.99\textwidth,
	    height=0.38\textheight,
	    ylabel={LP and IP values},
      	symbolic x coords={FLIRT-AN,FLIRT,SGM-AN,SGM,SLT,DDZ-AN,DDZ,VIRM-AN,VIRM,ICM},
	    x tick label style={rotate=90, font=\tiny},
	    xtick=data,
	    legend style={legend pos=north west,legend columns=1, font=\tiny},
	    ]
	    \addplot+[only marks, mark=-, error bars/.cd, y dir=minus, y explicit]
	        table [x=x,y=y,y error=error] {
	        x           y       error
	       	FLIRT-AN	18992.2	186.9
	        FLIRT	    17074.4	0
	       	SGM-AN	    19646.8	16.5
	        SGM	        18998.6	0
	       	SLT	        35291.7	0
	        DDZ-AN	    20040.2	0
	        DDZ	        18427.4	0
	        VIRM-AN	    98732.7	435.3
	        VIRM	    86947.4	6.4
	        ICM	        39885.7	34.3
	    };
	    \addplot+[cyan, only marks, mark=-,error bars/.cd, y dir=minus, y explicit]
	        table [x=x,y=y,y error=error] {
	        x           y       error
	       	FLIRT-AN	18306.2	0
	        FLIRT	    16820.8	0
	       	SGM-AN	    18939.8	15.4
	       	SGM	        18299.7	8.9
	        SLT	        34392.1	52.8
	        DDZ-AN	    19580.2	53.8
	        DDZ	        17957.4	43.8
	        VIRM-AN	    96503.9	1027.1
	        VIRM	    84459.2	430.5
	        ICM	        37011.4	388
	    };
	    \addplot+[green, only marks, mark=-, error bars/.cd, y dir=minus, y explicit]
	        table [x=x,y=   y,y error=error] {
	        x           y       error
	        ICM	        37032.2	391.0
	    };
	    \legend{$H\cdot$/$C$, $h\cdot$, $hA$}
	  \end{axis} 
	  \end{tikzpicture}
	  \caption{LP and IP values (bottom and top bars).}
	  \label{fig:LPIP data}
  \end{subfigure}
  \hfill
  \begin{subfigure}[b]{0.49\textwidth}
	  \centering
	  \begin{tikzpicture}
	  \begin{axis} [%
	    width=0.99\textwidth,
	    height=0.38\textheight,
	    ylabel={Relative gaps in percent},
	    ylabel shift=0.8em,
      	symbolic x coords={FLIRT-AN,FLIRT,SGM-AN,SGM,SLT,DDZ-AN,DDZ,VIRM-AN,VIRM,ICM},
	    x tick label style={rotate=90, font=\tiny},
	    xtick=data,
	    legend style={legend pos=north west,legend columns=1, font=\tiny},
	    ]
	    \addplot+[only marks, mark=-, error bars/.cd, y dir=plus, y explicit]
	        table [x=x,y=y,y error=error] {
	        x           y       error
	        FLIRT-AN	0       0.98
	        FLIRT	    0       0.00
	       	SGM-AN	    0       0.08
	        SGM	        0       0.00
	        SLT	        0       0.00
	        DDZ-AN	    0       0.00
	        DDZ	        0       0.00
	        VIRM-AN	    0       0.44
	        VIRM	    0       0.01
	        ICM	        0       0.09
	    };
	    \addplot+[cyan, only marks, mark=-,error bars/.cd, y dir=plus, y explicit]
	        table [x=x,y=y,y error=error] {
	        x           y       error
	       	FLIRT-AN	3.61	0.00
	        FLIRT	    1.49	0.00
	        SGM-AN	    3.60	0.08
	        SGM	        3.68	0.05
	        SLT	        2.55	0.15
	        DDZ-AN	    2.30	0.27
	        DDZ	        2.55	0.24
	        VIRM-AN	    2.26	1.04
	        VIRM	    2.86	0.50
	        ICM	        7.21	0.97
	    };
	    \addplot+[green, only marks, mark=-, error bars/.cd, y dir=plus, y explicit]
	        table [x=x,y=y,y error=error] {
	        x           y       error
	        ICM	        7.15	0.98
	    };
	    \legend{$H\cdot$/C,$h\cdot$, $hA$}
	  \end{axis} 
	  \end{tikzpicture}
	  \caption{Relative gaps \mbox{w.r.t.} to $\nu_\IP(\C)=\nu_\IP(\textrm{H}\cdot)$.}
	  \label{fig:relative gaps}
  \end{subfigure}
  \caption{Comparing optimal solution values for all models and instances.}
  \label{fig:data}
\end{figure}
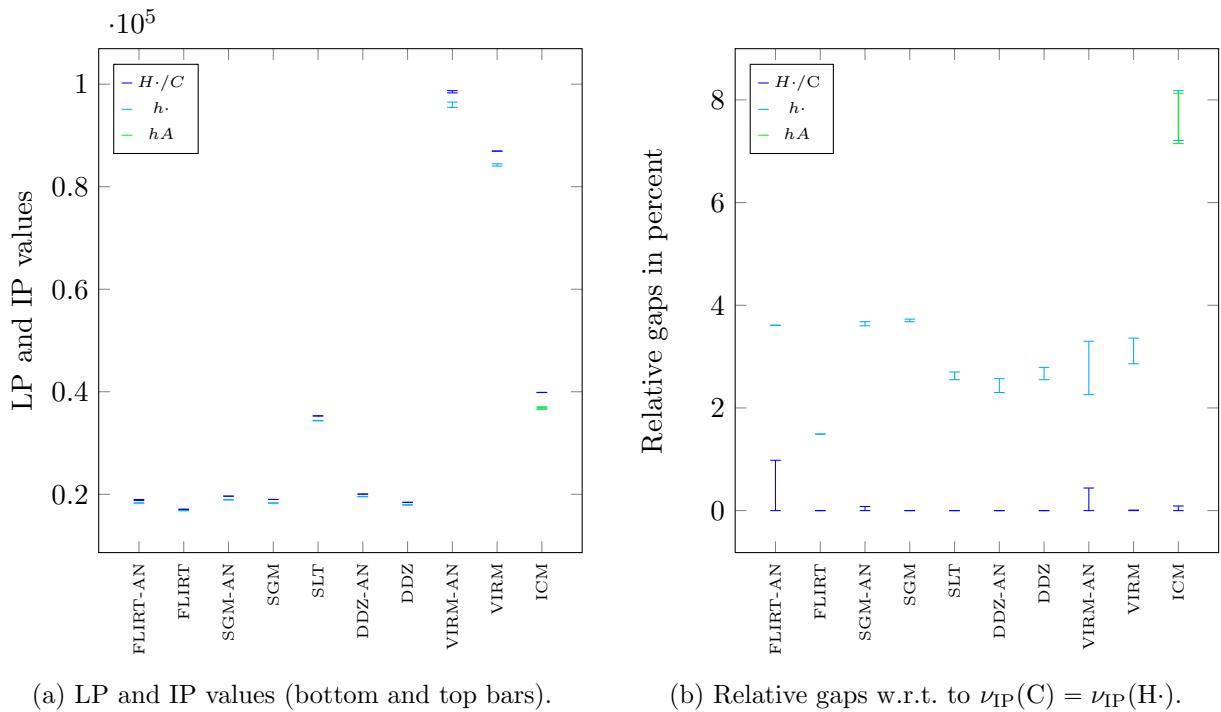

Our results show that the \CompositionModel{} and the full \HypergraphModel{}s (\HA{} and \HD{}) produce the same IP and LP relaxation values for all instances, see the blue values/bars.
This equality between the obtained values for \C{} and \HD{} directly follows from \autoref{prop:inclusions}, and so does the equality between \HA{} and the other models, as the NS instances generally satisfy $A = \bar{A}$.
One can also see that the integrality gap is very small for all instances for these models, where in absolute terms the difference is often smaller than the cost of 1-2 extra shunting actions.
In particular, the optimality gap is zero for six of the ten instances (FLIRT, SGM, SLT, DDZ-AN, DDZ, and VIRM), showing that the LP relaxation gives a very tight approximation of the integer program.
Some differences between the instances are revealed though by looking at the relative integrality gaps. 
FLIRT-AN has the largest integrality gap of $1\%$, followed by VIRM-AN with $0.44\%$, the remaining instances have gaps below $1\permille$.
These differences generally seem to be very instance specific, where larger integrality gaps are especially seen when mixing fractional compositions (e.g., a half red and a half blue composition) leads to a better match between capacity and demand for a given set of trips.

When looking at the small \HypergraphModel{}s (cyan values/bars), it can be seen that these always produce weaker IP and LP results.
This means that these models produce IP solutions that are not feasible within the \CompositionModel{}, indicating that the inequalities in \autoref{prop:inclusions} are indeed often strict for real-life instances.
Looking at the relative gaps reported in \autoref{fig:relative gaps}, it can be seen that the relative gap to the other models is especially big for the ICM instance.
This is also the only instance for which we see a difference between the values of the \hD{} and \hA{} models.
This higher gap and difference between the models can likely be explained by the longer compositions that can be formed in this instance, allowing more composition changes to be formed that are illegal in the other models.
Looking at the size of the optimality gaps, it is harder to make a comparison to the other models.
While the optimality gap is clearly larger for the ICM instance for both \hA{} and \hD{}, an integer LP relaxation solution is found by both models for the instance FLIRT.

\begin{table}[htbp]
	\caption{Cost components of the solutions to the linear programming relaxation (\textit{LP}) and integer program (\textit{MILP}) of the Composition and Hypergraph models for each of the NS instances.}
	\label{tab: cost components}
	\begin{tabularx}{\textwidth}{Xrrrrrrrr}
		\toprule
		& \multicolumn{4}{c}{$\hA{},\hD{}$} & \multicolumn{4}{c}{$\C{}, \HA{}, \HD{}$} \\
		\cmidrule(rl){2-5} \cmidrule(rl){6-9}
		& \multicolumn{2}{c}{LP} & \multicolumn{2}{c}{MILP}&  \multicolumn{2}{c}{LP} & \multicolumn{2}{c}{MILP} \\
		\cmidrule(rl){2-3} \cmidrule(rl){4-5} \cmidrule(rl){6-7} \cmidrule(rl){8-9}
		Instance   & Comp. & Coup. &  Comp. & Coup.& Comp. & Coup. &  Comp.  & Coup.\\
		\midrule
		FLIRT-AN & 18256.2 &  50.0 &  18256.2 &  50.0 & 18535.4 &  270.0 & 18722.2 &  270.0\\
		FLIRT   & 16760.8 &  60.0 &  16760.8 &  60.0 & 16784.4 &  290.0 & 16784.4 &  290.0\\
		SGM-AN  & 18894.4 &  30.0 &  18909.8 &  30.0 & 19016.9 &  613.3 & 19026.8 &  620.0\\
		SGM     & 18250.8 &  40.0 &  18259.7 &  40.0 & 18378.6 &  620.0 & 18378.6 &  620.0\\
		SLT     & 34229.3 & 110.0 &  34282.1 & 110.0 & 34451.7 &  840.0 & 34451.7 &  840.0\\
		DDZ-AN  & 19516.4 &  10.0 &  19570.2 &  10.0 & 19570.2 &  470.0 & 19570.2 &  470.0\\
		DDZ     & 17913.6 &   0.0 &  17957.4 &   0.0 & 17957.4 &  470.0 & 17957.4 &  470.0\\
		VIRM-AN & 95306.7 & 170.0 &  96323.9 & 180.0 & 96254.2 & 2043.2 & 96712.7 & 2020.0\\
		VIRM    & 83858.7 & 170.0 &  84289.2 & 170.0 & 84921.0 & 2020.0 & 84957.4 & 1990.0\\
		ICM     & 36530.0 & 111.2 &  36922.2 & 110.0 & 38080.1 & 1771.3 & 38095.7 & 1790.0\\
		\bottomrule
	\end{tabularx}
\end{table}

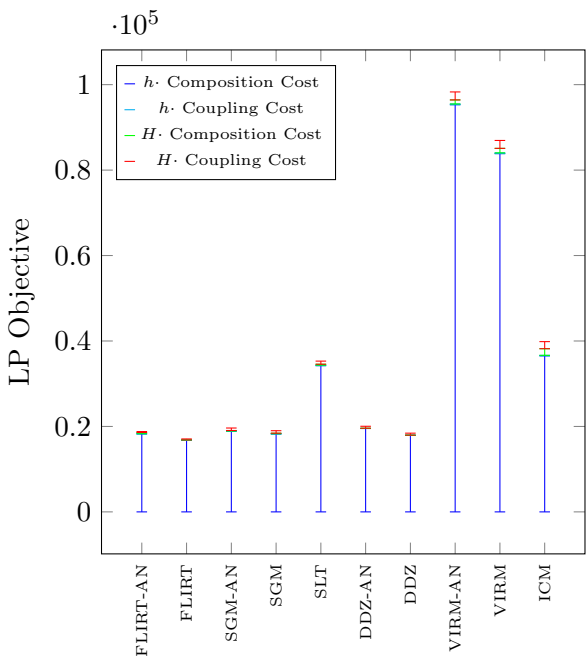
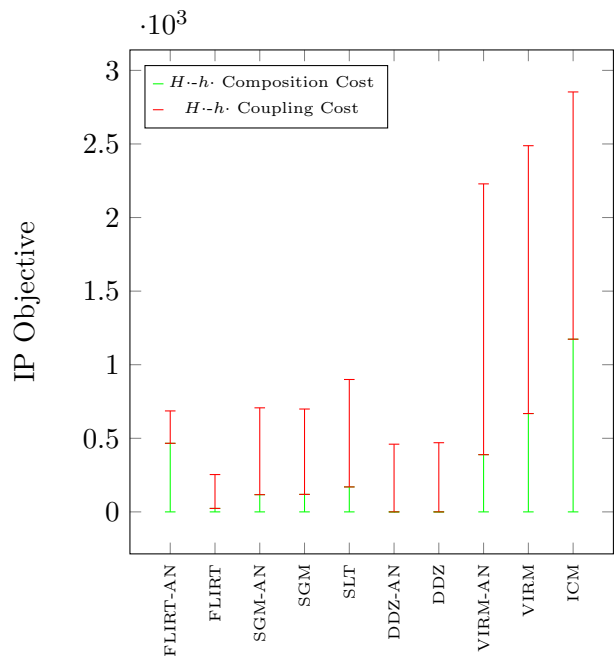
\begin{figure}[htbp]
	\centering%
	\begin{subfigure}[b]{0.49\textwidth}
		\centering
		\begin{tikzpicture}
			\begin{axis} [%
				width=0.99\textwidth,
				height=0.38\textheight,
				ylabel={LP Objective},
				symbolic x coords={FLIRT-AN,FLIRT,SGM-AN,SGM,SLT,DDZ-AN,DDZ,VIRM-AN,VIRM,ICM},
				x tick label style={rotate=90, font=\tiny},
				xtick=data,
				legend style={legend pos=north west,legend columns=1, font=\tiny},
				]
				\addplot+[only marks, mark=-, error bars/.cd, y dir=plus, y explicit]
				table [x=x,y=y,y error=error] {
					x           y       error
					DDZ-AN	    0       19516.4
					DDZ	        0       17913.6
					FLIRT-AN	0       18256.2
					FLIRT	    0       16760.8
					SGM-AN	    0       18894.4
					SGM	        0       18250.8
					VIRM-AN	    0       95306.7
					VIRM	    0       83858.7
					SLT	        0       34229.3
					ICM	        0       36530.0
				};
				\addplot+[cyan, only marks, mark=-,error bars/.cd, y dir=plus, y explicit]
				table [x=x,y=y,y error=error] {
					x           y       error
					DDZ-AN	    19516.4    10
					DDZ	        17913.6     0
					FLIRT-AN	18256.2    50
					FLIRT	    16760.8    60
					SGM-AN	    18894.4    30
					SGM	        18250.8    40
					VIRM-AN	    95306.7   170
					VIRM	    83858.7   170
					SLT	        34229.3   110
					ICM	        36530.0   111.2
				};
				\addplot+[green, only marks, mark=-, error bars/.cd, y dir=plus, y explicit]
				table [x=x,y=   y,y error=error] {
					x           y       error
					DDZ-AN	    19526.4	53.8
					DDZ	        17913.6 43.8
					FLIRT-AN	18306.2	279.2
					FLIRT	    16820.8	23.6
					SGM-AN	    18924.4	122.5
					SGM	        18290.8	127.8
					VIRM-AN	    95476.7	947.5
					VIRM	    84028.7	1062.3
					SLT	        34339.3	222.4
					ICM	        36641.2 1550.1
				};
				\addplot+[red, only marks, mark=-, error bars/.cd, y dir=plus, y explicit]
				table [x=x,y=   y,y error=error] {
					x           y       error
					DDZ-AN	    19580.2	460
					DDZ	        17957.4 470
					FLIRT-AN	18585.4	220
					FLIRT	    16844.4	230
					SGM-AN	    19046.9	583.3
					SGM	        18418.6	580
					VIRM-AN	    96424.2	1873.2
					VIRM	    85091.0	1850
					SLT	        34561.7	730
					ICM	        38191.3 1660.1
				};
				\legend{ $h\cdot$ Composition Cost, $h\cdot$ Coupling Cost, $H\cdot$ Composition Cost, $H\cdot$ Coupling Cost}
			\end{axis} 
		\end{tikzpicture}
		\caption{Cost components of $\nu_\IP(H\cdot)$ and $\nu_\IP(\textrm{h}\cdot)$.}
		\label{fig:LPCompData}
	\end{subfigure}
	\hfill
	\begin{subfigure}[b]{0.49\textwidth}
		\centering
		\begin{tikzpicture}
			\begin{axis} [%
				width=0.99\textwidth,
				height=0.38\textheight,
				ylabel={IP Objective},
				ylabel shift=0.8em,
				symbolic x coords={FLIRT-AN,FLIRT,SGM-AN,SGM,SLT,DDZ-AN,DDZ,VIRM-AN,VIRM,ICM},
				x tick label style={rotate=90, font=\tiny},
				xtick=data,
				scaled y ticks=base 10:-3,
				legend style={legend pos=north west,legend columns=1, font=\tiny}
				]
				\addplot+[green, only marks, mark=-, error bars/.cd, y dir=plus, y explicit]
				table [x=x,y=y,y error=error] {
					x           y       error
					FLIRT-AN	0       466.0
					FLIRT	    0       23.6
					SGM-AN	    0       117.0
					SGM	        0       118.9
					SLT	        0       169.6
					DDZ-AN	    0       0.00
					DDZ	        0       0.00
					VIRM-AN	    0       388.8
					VIRM	    0       668.2
					ICM	        0       1173.5
				};
				\addplot+[red, only marks, mark=-,error bars/.cd, y dir=plus, y explicit]
				table [x=x,y=y,y error=error] {
					x           y       error
					FLIRT-AN	466.0     220
					FLIRT	    23.6      230
					SGM-AN	    117.0     590
					SGM	        118.9     580
					SLT	        169.6     730
					DDZ-AN	    0.00      460
					DDZ	        0.00      470
					VIRM-AN	    388.8     1840
					VIRM	    668.2     1820
					ICM	        1173.5    1680
				};
				\legend{$H\cdot$-$h\cdot$ Composition Cost, $H\cdot$-$h\cdot$ Coupling Cost}
			\end{axis} 
		\end{tikzpicture}
		\caption{Cost components of gap $\nu_\IP(H\cdot) - \nu_\IP(\textrm{h}\cdot)$.}
		\label{fig:h-H-IP-gap}
	\end{subfigure}
	\caption{Analyzing the cost components of $\nu_{M}(h\cdot)$ and $\nu_{M}(\textrm{H}\cdot)$.}
	\label{fig:CompData}
\end{figure}

Further insight into how the small and full \HypergraphModel{}s compare to each other is given by \autoref{tab: cost components}, where the different cost components are shown, i.e., composition and coupling costs for each solution of the respective model. 
These results are visualized in \autoref{fig:LPCompData} and \autoref{fig:h-H-IP-gap}.
Here, \autoref{fig:LPCompData} gives the composition and coupling cost for the small \HypergraphModel{}s and, on top of that, the additional composition and coupling costs in the full \HypergraphModel{}s.
The difference is further investigated in \ref{fig:h-H-IP-gap}, where the distribution of additionally required composition and coupling cost for a solution to the full \HypergraphModel{}s compared to the small models is given.
Two conclusions can be drawn immediately. 
The first is that the LP relaxation's coupling cost is a very sound approximation of the IP's coupling cost as it differs only in 5 out of 20 cases, with a maximum absolute deviation of only 30. 
The second is that the small \HypergraphModel{}s underestimate the coupling costs resulting from the \CompositionModel{}, respectively, the full \HypergraphModel{}s by a significant amount. 
Indeed, the solutions to the small \HypergraphModel{}s spend at most $\approx3.2\permille$ for coupling while the solutions to the full models spend at least $\approx2.4\%$ for coupling. 
Regarding the composition costs required to operate an instance, it can be seen that the smaller models do give a sound approximation.

When looking at the found integrality gaps, two instances stand out. 
FLIRT-AN is the instance with the largest integrality gap for the full \HypergraphModel{}s while there is no integrality gap for the small models. 
Further analysis showed that this larger gap is particularly the result of two sequences of connected trips, which both contain a trip during the rush hour with a high seat shortage when operated in a single train unit composition.
Unless uncoupling can be done afterwards, these trips would force the whole sequences to be operated in a double train unit composition.
It turns out that such uncoupling actions are not optimal in the \C{}, \HA{}, and \HD{} models due to the train unit ending inventories.
However, their LP relaxations can operate these sequences with a fractional train unit composition, thus resulting in significantly lower composition costs for the LP models and a relatively large optimality gap.
The second instance that stands out in \autoref{fig:relative gaps} is VIRM-AN.
A closer examination reveals that the solutions to the LP-relaxations of the full and small \HypergraphModel{}s require different amounts of coupling, which leads to a larger amount of deviating compositions between the LP and the IP solutions.
This explains the gap between the objective function values.

\paragraph{Computation Time}
Lastly, we look at the computation time required to solve each of the models and their corresponding linear programming relaxation.
The results for each of the NS instances are given in \autoref{tab: solution times} and visualized in \autoref{fig:runtimes}. 
The computations were performed on an Intel\textsuperscript{\textregistered} Core(TM) i7-9700K CPU @ 3.60GHz with 8 cores and 8 threads.

\begin{table}[htbp]
	\caption{Computation time in seconds for solving the linear programming relaxation (\textit{LP}) and integer program (\textit{MILP}) of the Composition and Hypergraph models for each of the NS instances.}
	\label{tab: solution times}
\begin{tabularx}{\textwidth}{Xrrrrrrrrrr}
\toprule
& \multicolumn{2}{c}{C} & \multicolumn{2}{c}{\HD} & \multicolumn{2}{c}{\hD} & \multicolumn{2}{c}{\HA} & \multicolumn{2}{c}{\hA} \\
\cmidrule(rl){2-3} \cmidrule(rl){4-5} \cmidrule(rl){6-7} \cmidrule(rl){8-9} \cmidrule(rl){10-11}
instance & LP & MILP & LP & MILP & LP & MILP & LP & MILP & LP & MILP \\
\midrule
FLIRT-AN &   0.02 &   0.09 &   0.05 &   0.12 &   0.17 &   0.33 &   0.05 &   0.38 &   0.09 &   0.25 \\
FLIRT    &   0.02 &   0.09 &   0.03 &   0.10 &   0.10 &   0.33 &   0.03 &   0.40 &   0.10 &   0.25 \\
SGM-AN   &   0.09 &   0.33 &   0.18 &   0.41 &   0.31 &   0.92 &   0.27 &   1.49 &   0.32 &   1.43 \\
SGM      &   0.07 &   0.21 &   0.11 &   0.31 &   0.30 &   0.82 &   0.27 &   1.01 &   0.30 &   0.93 \\
SLT      &   0.12 &   0.19 &   0.16 &   0.38 &   0.55 &   1.43 &   0.32 &   1.55 &   0.59 &   1.87 \\
DDZ-AN   &   0.05 &   0.10 &   0.12 &   0.26 &   0.11 &   0.54 &   0.12 &   1.07 &   0.14 &   0.94 \\
DDZ      &   0.25 &   0.11 &   0.05 &   0.22 &   0.12 &   0.44 &   0.13 &   1.07 &   0.13 &   0.71 \\
VIRM-AN  &   0.35 &   1.50 &   0.73 &   2.14 &   1.09 &   5.16 &   0.97 &   8.74 &   0.50 &   7.84 \\
VIRM     &   0.34 &   0.99 &   0.73 &   1.41 &   0.86 &   5.54 &   1.02 &   7.59 &   0.56 &   5.49 \\
ICM      &   2.92 &   4.28 &   5.45 &   7.35 &   1.89 &  35.18 &  93.86 & 2502.83 &   3.21 & 143.28 \\
\bottomrule
\end{tabularx}
\end{table}

\begin{figure}[htbp]
  \centering%
  \begin{subfigure}[b]{0.49\textwidth}
    \begin{tikzpicture}
    \begin{semilogyaxis}[
      width=0.99\textwidth,
      height=0.38\textheight,
      ylabel={LP run times in seconds},
      symbolic x coords={FLIRT-AN,FLIRT,SGM-AN,SGM,SLT,DDZ-AN,DDZ,VIRM-AN,VIRM,ICM},
      x tick label style={rotate=90, font=\tiny},
      xtick=data,
      legend style={legend pos=north west,legend columns=1, font=\tiny},
      ymin=.01, ymax=10000,
    ]
      \addplot[blue, mark=*, mark size=1pt] table[x=instance,y=C-LP,col sep=comma]{tables/runtime_table.csv};
      \addplot[orange, mark=*, mark size=1pt] table[x=instance,y=HD-LP,col sep=comma]{tables/runtime_table.csv};
      \addplot[cyan, mark=*, mark size=1pt] table[x=instance,y=hD-LP,col sep=comma]{tables/runtime_table.csv};
      \addplot[red, mark=*, mark size=1pt] table[x=instance,y=HA-LP,col sep=comma]{tables/runtime_table.csv};
      \addplot[green, mark=*, mark size=1pt] table[x=instance,y=hA-LP,col sep=comma]{tables/runtime_table.csv};
      \legend{$\nu_\LP(\C)$, $\nu_\LP(\HD)$, $\nu_\LP(\hD)$, $\nu_\LP(\HA)$, $\nu_\LP(\hA)$}
    \end{semilogyaxis}
    \end{tikzpicture}
    \caption{LP run times.}
    \label{fig:LP runtime}
  \end{subfigure}
  \hfill
  \begin{subfigure}[b]{0.49\textwidth}
    \begin{tikzpicture}
    \begin{semilogyaxis}[
      width=0.99\textwidth,
      height=0.38\textheight,
      ylabel={IP run times in seconds},
      symbolic x coords={FLIRT-AN,FLIRT,SGM-AN,SGM,SLT,DDZ-AN,DDZ,VIRM-AN,VIRM,ICM},
      x tick label style={rotate=90, font=\tiny},
      xtick=data,
      legend style={legend pos=north west,legend columns=1, font=\tiny},
      ymin=.01, ymax=10000,
    ]
      \addplot[blue, mark=*, mark size=1pt] table[x=instance,y=C-IP,col sep=comma]{tables/runtime_table.csv};
      \addplot[orange, mark=*, mark size=1pt] table[x=instance,y=HD-IP,col sep=comma]{tables/runtime_table.csv};
      \addplot[cyan, mark=*, mark size=1pt] table[x=instance,y=hD-IP,col sep=comma]{tables/runtime_table.csv};
      \addplot[red, mark=*, mark size=1pt] table[x=instance,y=HA-IP,col sep=comma]{tables/runtime_table.csv};
      \addplot[green, mark=*, mark size=1pt] table[x=instance,y=hA-IP,col sep=comma]{tables/runtime_table.csv};
      \legend{$\nu_\IP(\C)$, $\nu_\IP(\HD)$, $\nu_\IP(\hD)$, $\nu_\IP(\HA)$, $\nu_\IP(\hA)$}
    \end{semilogyaxis}
    \end{tikzpicture}
    \caption{IP run times.}
    \label{fig:IPruntime}
  \end{subfigure}
  \caption{Comparing computation times for all models and instances.}
  \label{fig:runtimes}
\end{figure}
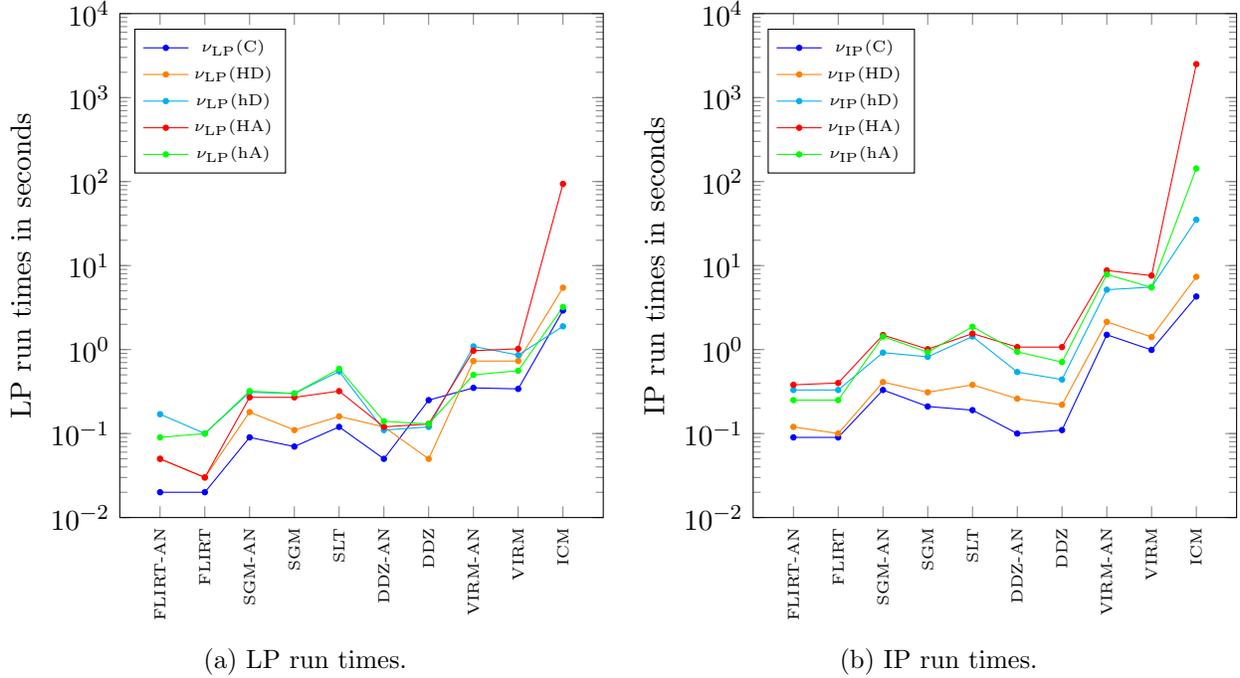

Most instances can be solved quickly by all of the models, where the computation time required to solve the linear programming relaxation and integer program is below 10 seconds for all instances except ICM.
Especially the sprinter instances FLIRT(-AN), SGM(-AN), and SLT can be solved quickly, with a computation time of less than 2 seconds for all of the models.
The instance that stands out as the hardest is the intercity ICM instance, for which a longer computation time can both be seen when solving the linear programming relaxations and integer programs.
This can likely be explained by the larger number of compositions possible for this instance.
Especially the longer computation time of the \HA{} model should be noted, which is in line with the large size of this model for the ICM instance.

When comparing the models, it can be seen that the \CompositionModel{} has the lowest computation time for solving the integer program for all instances, and often has one of the lowest solving times when solving the linear programming relaxation.
Closest in computation time for the integer program is the full \HypergraphModel{} \HD{}.
That this model also performs well is likely due to it providing the same linear programming bound as the \CompositionModel{} and its relatively compact size in terms of variables, even though it is significantly larger in terms of the number of constraints.
It can be seen that the small \HypergraphModel{}s often perform well in terms of solving the linear programming relaxation, but lack behind more significantly for the integer programs.
This is especially the case for the larger ICM instance.
The computation times of the non-contracted model \HA{} are overall the longest, which is in line with the larger size of this model.
As noted before, especially the long solution time for the ICM instance is notable for this model.

Another interesting observation from our results is the relatively large computation time of the linear programming relaxation compared to the full integer program.
This can be explained by the results we found on the strength of the linear programming relaxation, meaning that branching is often not or only to a limited degree required when solving the integer program.
The only exception to this is for the ICM instance, for which we also see a significantly larger gap for some of the models between the time needed to solve the linear programming relaxation and the integer program.

\paragraph{Discussion}
Overall, the numerical results show that the \CompositionModel{} performs best for our test instances.
It has the lowest computation time for solving the integer program for all instances and has an LP relaxation computation time that is among the best for most instances.
Moreover, it provides the same linear programming bound as the full \HypergraphModel{}s but is more compact than these models, in terms of both the number of constraints and variables.
This difference in size, and computation time, is especially apparent when comparing to the \HD{} model for the larger ICM instance.
While the small \HypergraphModel{}s are generally also compact in size, our results show that these models provide integer solutions that are not feasible in practice for the NS instances.

\section{Conclusion}
\label{sec: conclusion}

In this paper, we compared models that have been proposed for the scheduling of rolling stock at passenger railway operators.
Through a literature review, we have shown how models differ due to differences in the operational settings of railway operators, particularly leading to differences in the handling of compositions, turnings, maintenance requirements, and passenger demand.
Our literature review also shows that a number of clear streams can be identified in the literature, which either make use of the same core model or focus on the same problem setting.

In an analytic comparison, we further analyzed the core models used in two streams of the literature:  the \CompositionModel{} that has been proposed for the setting of Netherlands Railways (NS) and the \HypergraphModel{} that has been proposed for the setting of DB Fernverkehr AG (DB).
In our theoretical analysis, we focused on the rolling stock scheduling setting of NS, suggesting different variants of the \HypergraphModel{} to vary the compactness and expressiveness of this model.
We show that the linear programming relaxation provided by the \CompositionModel{} is not necessarily the tightest of all models, but that the linear programming relaxations of all high quality models are equivalent in many cases of practical relevance, even if not all sufficient  equivalence conditions are met.

Lastly, we have compared the performance of the \CompositionModel{} and the different variants of the \HypergraphModel{} numerically on real-world instances of NS, which represent different rolling stock sub-fleets.
Our analysis shows that the \CompositionModel{} combines compactness, especially in terms of the number of variables, with a very tight linear programming bound.
These results are also reflected in the computation times, which are overall the lowest for the \CompositionModel{}.
However, some of the variants of the \HypergraphModel{} also solve quickly on all instances. 
Clear differences exist though between the \HypergraphModel{} variants.

Our results show that many factors are of importance when solving rolling stock scheduling problems.
Solving such problems thus requires a careful selection and tuning of a most appropriate model, where the appropriateness of a model depends strongly on the operational setting.
Our analysis particularly shows the analytic and numerical relations between two commonly used models within the context of the NS setting.
A similar analysis could clarify the relations between other rolling stock scheduling models, and study different operational contexts.

\section{Acknowledgments}
This work has been partly supported by the Research Campus MODAL (Mathematical Optimization and Data Analysis Laboratories), funded by the Federal Ministry of Education and Research (BMBF Grant 05M14ZAM). Moreover, we would like to thank Erasmus Trustfonds for partly funding Rowan Hoogervorst's research visit to Zuse Institute Berlin (ZIB), during which this work was started.

\bibliographystyle{informs2014trsc}
\bibliography{bibliography}

\newpage
\appendix

\end{document}